\documentclass{siamltex}

\date{\today}

\usepackage{amsfonts,amsmath,amssymb,psfrag,graphicx,bbm,mathrsfs,yhmath,color}
\usepackage{hyperref}
\usepackage[left=2cm,right=2cm,vmargin=2.5cm,footnotesep=0.5cm]{geometry}

\usepackage[normalem]{ulem}

\def\fin{\ifmmode{\Large$\diamond$}\else{\unskip\nobreak\hfil
    \penalty50\hskip1em\null\nobreak\hfil{\Large$\diamond$}
    \parfillskip=0pt\finalhyphendemerits=0\endgraf}\fi}

\DeclareMathOperator*{\esssup}{ess\,sup}
\def\be#1#2\ee{\begin{equation}\label{eq:#1}#2\end{equation}}
\def\req#1{{\rm(\ref{eq:#1})}}
\def\bdm  {\begin{displaymath}}
  \def\edm  {\end{displaymath}}
\def\bdmal{\begin{displaymath}\begin{aligned}}
    \def\edmal{\end{aligned}\end{displaymath}}

\mathchardef\Omega="010A
\mathchardef\PhiG="0108

\renewcommand{\L}{{\mathscr L}}
\newcommand{\N}{{\mathord{\mathbb N}}}
\newcommand{\R}{{\mathord{\mathbb R}}}

\newcommand{\Z}{{\mathbb{Z} }}

\newcommand{\E}{{\mathbb E}}
\newcommand{\Var}{{\mathbb V}\textnormal{ar}}
\newcommand{\Cov}{{\mathbb C}\textnormal{ov}}

\renewcommand{\L}{{\mathscr L}}

\newcommand{\norm}[1]{\|#1\|}

\newcommand{\rmd}{\,\mathrm{d}}
\newcommand{\ds}{\rmd s}
\newcommand{\dt}{\rmd t}
\newcommand{\dx}{\rmd x}

\newcommand{\dxi}{\rmd \xi}

\newcommand{\rmi}{\mathrm{i}}
\newcommand{\eps}{\varepsilon}

\newcommand{\omegahat}{{\widehat{\omega}}}
\newcommand{\fhat}{{\widehat{f}}}

\newcommand{\testfct}{v}
\newcommand{\Testfct}{V}
\newcommand{\testfctphi}{\phi}
\newcommand{\Ghat}{{\widehat \testfct}}
\newcommand{\Khat}{{\widehat K}}

\newcommand{\PP}{\ensuremath{\mathsf{P}}}
\newcommand{\LL}{\ensuremath{\mathsf{L}}}
\newcommand{\KK}{\ensuremath{\mathsf{K}}}

\newcommand{\GG}{\ensuremath{\mathsf{G}}}

\def\req#1{{\rm(\ref{eq:#1})}}

\makeatletter
\newcommand{\dupdots}{\mathinner{\mkern1mu\raise\p@
    \vbox{\kern7\p@\hbox{.}}\mkern2mu
    \raise4\p@\hbox{.}\mkern2mu\raise7\p@\hbox{.}\mkern1mu}}
\makeatother
\makeatletter
\newcount\c@MaxMatrixCols \c@MaxMatrixCols=10

\makeatother

 {\endMakeFramed}

\newcommand{\chiL}{{\mathrm{1}_\Lambda}}

\newcommand{\V}{{{\mathscr V}_{u_0}}}

\newcommand{\Y}{{\mathscr Y}}

\newcommand{\dr}{\rmd r}
\newcommand{\dy}{\rmd y}

\newcommand{\xx}{{\boldsymbol{x}}}
\newcommand{\dxx}{\rmd \xx}
\newcommand{\yy}{{\boldsymbol{y}}}
\newcommand{\dyy}{\rmd \yy}
\newcommand{\xifett}{{\boldsymbol{\xi}}}
\newcommand{\dxifett}{\rmd \xifett}

\renewcommand{\mid}{\,|\,}

\makeatletter
\renewcommand\@biblabel[1]{#1.}
\makeatother

\title{(Non-)Hyperuniformity of Second Order Statistics\\ 
of Point Processes\thanks{The research leading to this work has been 
    done within the Collaborative Research Center TRR 146; 
    corresponding funding by the DFG is gratefully acknowledged.}}
\author{Fabio Frommer\thanks{Institut f\"ur Mathematik, Johannes
    Gutenberg-Universit\"at Mainz, 55099 Mainz, Germany
    ({\tt fabiofrommer@uni-mainz.de}).} \and
    Martin Hanke\thanks{Institut f\"ur Mathematik, Johannes
    Gutenberg-Universit\"at Mainz, 55099 Mainz, Germany
    ({\tt hanke@math.uni-mainz.de}).} }

\begin{document}

\sloppy
\maketitle

\begin{abstract}
We investigate statistical properties of certain stationary point
processes, namely determinantal processes with projection kernels
and Gibbs point processes with superstable pair interactions. These 
are examples of hyperuniform and non-hyperuniform stationary 
point processes, respectively. 
We are interested in the variance of their second order statistics
within a ball around the origin, and we study the asymptotic growth of 
this variance as the radius of the ball goes to infinity.
It is shown that, generically, for both types of processes the variance 
is asymptotically proportional to the volume of the ball.
In other words: the second order statistics of these point
processes behave non-hyperuniform.
For Gibbs processes with superstable interactions these results have
an interesting application to the so-called
inverse Henderson problem of statistical mechanics.

We also show that the structure factor (respectively the 
Bartlett spectral measure) of these Gibbs processes is strictly positive,
while it is positive except for a simple zero at the origin
for the determinantal processes.
\end{abstract}

\begin{AMS}
  {\sc 82B21, 82B80, 60G55}
\end{AMS}


\pagestyle{myheadings}
\thispagestyle{plain}

\addtocounter{footnote}{1}

\section{Introduction}
Stationary, i.e., translation invariant point processes 
whose density fluctuations fail to be extensive quantities, 
are called \emph{hyperuniform}. 
Loosely speaking, hyperuniform systems exhibit an increasing amount of
``order'' when looking at increasing scales: 
although the variance of the number of points in a bounded set may be positive, it does not grow as fast as the volume of that set.
In the physics community this phenomenon is described as \emph{global order and local disorder}.
In contrast, a realization of a Poisson point process 
-- which is non-hyperuniform -- will exhibit point clusters and empty regions of any size. 

Hyperuniformity was popularized by Torquato and Stillinger in \cite{Torquato03} under this name;
other works, e.g., by Ghosh and Lebowitz~\cite{Ghosh17} have used the term \emph{superhomogeneity} instead.
In statistical physics, hyperuniform fluids are called \emph{incompressible}.
Not surprisingly, this concept has many applications in different fields of material science, chemistry, physics, and biology;
compare the review of Torquato~\cite{Torquato18}.
In the mathematics community hyperuniform point processes arise in the 
context of Coulomb systems~\cite{Kunz74,Lebowitz83,Martin80},  
zeros of Gaussian analytic functions~\cite{Forrester99}, 
eigenvalues of random matrices \cite{Forrester10}, 
quasi-crystals~\cite{Baake19,Bjoerklund24,Torquato17a}, 
perturbed lattices~\cite{Dereudre24}, 
and certain determinantal point processes \cite{Ghosh21}.

For a stationary point process hyperuniformity can be inferred from the 
structure factor of the system,
i.e., from correlations of the system and its shifts:
For a hyperuniform system the structure factor vanishes at the origin. One can 
attempt to quantify the \emph{local disorder} of the system somewhat further
by looking, for example, at the number of pairs of points within a given 
distance. 
This cannot be easily resolved by the structure factor; 
instead one can resort to higher order statistics.
In this paper we therefore study \emph{second order statistics}, and 
we investigate whether these are extensive quantities. It is easy to see that 
this is the case for the Poisson point process, whereas, for example,
the stationary lattice exhibits hyperuniformity also for second order 
statistics. (Some technical difficulties arise due to the lack of Fourier 
smoothness of certain indicator functions; compare Remark~\ref{Rem:Lotz26} below.)

We focus on two classes of point processes:
\begin{itemize}
    \item[(i)] determinantal point processes with a projection kernel, 
    which are known to be hyperuniform, cf.~Ghosh and Krishnapur~\cite{Ghosh21};
    \item[(ii)] Gibbs point processes with superstable pair interactions,
    which are known to be non-hyperuniform, cf.~Ruelle~\cite{Ruelle70}.
\end{itemize}
It turns out that for both of these systems the structure factor is positive
(except for a simple zero at the origin for the determinantal processes). 
Moreover, in both cases the variance of non-trivial second-order statistics 
is also an extensive quantity generically. 
Note that this insight can be combined with recent results
by Hirsch, Otto, and Svane~\cite{Svane25} to obtain central limit theorems for second order statistics
of Gibbs point processes.

As another application of our results we consider the so-called 
inverse Henderson problem~\cite{FHJ19,FrHa22}, which is concerned with 
the identification of the
pair interaction of a stationary Gibbs point process from measurements of 
the radial distribution function. In this application second order statistics
correspond to the values of quadratic Taylor approximations of
a maximum-entropy type functional, and non-hyperuniformity of these statistics 
implies that the corresponding approximation -- like the entropy functional itself --
is strictly convex. 

The outline of this paper is as follows. In Section~\ref{Sec:pointprocesses}
we review fundamental properties of the two aforementioned model processes 
(i.e., determinantal and Gibbs point processes) 
in as much as they are needed for our results. 
Section~\ref{Sec:Nonhyperuniformity} provides a rigorous definition of
hyperuniformity and the structure factor (or the Bartlett spectral measure,
respectively), and investigates the latter for our two model systems.
Second order statistics of these systems are subsequently treated in 
Section~\ref{Sec:2ndorder}, and the application to the inverse Henderson
problem is the subject of Section~\ref{Sec:Jacobian}. 
Finally, in Section~\ref{Sec:Auxiliaries} we collect some technical
results, whose derivations would have disturbed the basic flow of argument 
in the main body of the paper.
\color{black}

\section{The point processes under consideration}
\label{Sec:pointprocesses}
For a measurable subset $\triangle\subset\R^d$ denote by $|\triangle|$ its 
Lebesgue measure, and for a configuration $\gamma\subset\R^d$ let 
$\#\gamma\in\N_0\cup\{+\infty\}$ be the number of its elements. 
We consider the configuration space
\[
   \Gamma \,=\, \bigl\{\,\gamma\subset\R^d\,:\,\#(\gamma\cap\triangle)<\infty \
   \text{for every bounded $\triangle\subset\R^d$}\,\bigr\}
\]
and its subset
\[
   \Gamma_0 \,=\, \bigl\{\,\gamma\in\Gamma\,:\, \#\gamma<\infty\,\bigr\}
\]
of finite configurations, and
define the $\sigma$-algebra
\[
   \mathcal{F}
   \,=\, \sigma\bigl(\,
            \{\gamma\in\Gamma\,:\, \#(\gamma\cap\triangle) = m\} \,:\, 
            \Delta \subset \R^d \text{ bounded}, \, m \in \N_0\,\bigr)\,.
\]
Any probability measure $\PP$ on $(\Gamma,\mathcal{F})$ is called a 
\emph{point process}. 
We write $\E\bigl[\,\cdot\,\bigr]$, $\Var\bigl[\,\cdot\,\bigr]$,
and $\Cov\bigl[\,\cdot\,,\,\cdot\,\bigr]$ for the expectation, variance,
and covariance of random variables under $\PP$, respectively.
If there exist nonnegative functions $\rho^{(n)}:(\R^d)^n\to\R^+_0$, 
$n\in\N$, such that
\be{useofrho}
   \E\Bigl[\sum_{x_1\neq \dots \neq x_n\in\gamma}G(\xx_n)\Bigr] 
   \,=\, \int_{(\R^d)^n} G(\xx_n)\rho^{(n)}(\xx_n) \dxx_n
\ee
for every nonnegative function $G:(\R^d)^n\to[0,\infty]$ and every $n\in\N$ then these are called \emph{correlation functions} of $\PP$. 
If there exists a $q>0$ such that 
\be{R}
   \rho^{(n)}(\xx_n) \,\leq\, q^n \qquad
   \text{for all $\xx_n\in(\R^d)^n$ and all $n\in\N$}\,,
\tag{R}
\ee
then it is said that the correlation functions satisfy a \emph{Ruelle bound}; 
this condition determines $\PP$ uniquely, cf., e.g., Kuna~\cite{KunaPhD}. 

A point process is called \emph{translation invariant} or \emph{stationary}, 
if $\PP$ is invariant under $\tau_y$ for every
$y\in \R^d$ where $\tau_y \colon x\mapsto x+y$. 
If such a point process admits correlation functions then these inherit 
the invariance under $\tau_y$, when all variables undergo the same translation. 
It follows that the first correlation function $\rho^{(1)}$ of a 
stationary point process $\PP$ is constant, this constant being 
the \emph{density} or \emph{intensity} $\rho$ of $\PP$. 
Further, $\rho^{(2)}$ only depends on the difference $x-y$ of its two 
arguments $x,y\in\R^d$; we therefore utilize the even function
\[
   \rho_2(x) \,=\, \rho^{(2)}(x,0)\,, \qquad x\in\R^d\,,
\]
rather than $\rho^{(2)}$ in this case, and we also make use of the so-called 
\emph{truncated pair correlation function}
\be{Ursell}
   \omega(x) \,=\, \rho_2(x) \,-\, \rho^2\,, \qquad x\in\R^d\,.
\ee

The best known example for a stationary point process with correlation functions is 
the \emph{Poisson point process} with rate $\rho$; 
its (constant) correlation functions are given by
\be{rhopoiss}
    \rho^{(n)}(\xx_n) \,=\, \rho^n\,, \qquad n\in\N\,,
\ee
so they satisfy the Ruelle bound with $q=\rho$. 
In particular, the density of the Poisson point process is given by its rate, 
and the truncated pair correlation function is vanishing identically, 
i.e., $\omega = 0$.

Another example of a stationary point process is the 
\emph{stationary lattice} $\mathsf{L} = \Z^d +Y$, 
where the (constant) shift $Y$ is uniformly 
distributed in $[0,1)^d$. $\LL$ does not admit correlation functions 
$\rho^{(n)}$ with $n\geq 2$, but its density exists, i.e., 
$\rho^{(1)}=1$.

\subsection{Determinantal point processes with projection kernels}\label{ss:DPP}
Throughout this paper we (formally) denote by
\[
   \fhat(\xi) \,=\, \int_{\R^d}e^{-2\pi\rmi x\cdot\xi} f(x)\dx\,, \qquad
   \xi\in\R^d\,,
\]
the Fourier transform of a function $f\in L^2(\R^d)$. Given a symmetric
set $E\subset\R^d$ with Lebesgue measure $|E|=1$, let
\[
   K(x) \,=\, \int_E e^{2\pi\rmi x\cdot\xi}\dxi
\]
be the inverse Fourier transform of the characteristic function $1_E$.
Since $E$ is symmetric and has finite measure, 
$K$ is a real-valued even continuous function, and the
convolution integral operator 
\be{convolution}
   (\mathscr{K} \testfct)(x) = \int_{\R^d}K(x-y)\testfct(y) \dy
\ee
is the orthogonal projection onto all $L^2$-functions whose Fourier transforms
are supported in $E$. Moreover, since the Fourier transform is a unitary 
operator, there holds
\be{K-L2}
    \norm{K}_{L^2} \,=\, \norm{\widehat{K}}_{L^2}
    \,=\, \norm{1_E}_{L^2} \,=\, |E| \,=\, 1 \,=\, K(0)\,.
\ee

Associated with the operator $\mathscr{K}$ is a unique stationary
point process $\KK$, defined in terms of its correlation functions
\begin{align}\label{eq:rhodpp}
    \rho^{(n)}(\xx_n) = \det \bigl[K(x_i-x_j)\bigr]_{i,j=1}^n\,, \qquad
    \xx_n=(x_1,\dots,x_n)\in(\R^d)^n\,,
\end{align}
which satisfy the Ruelle bound with $q=1$;
see Soshnikov~\cite{Soshnikov} for further details. 
This process $\KK$ 
belongs to the family of \emph{determinantal point processes}; its density
is given by $\rho=K(0)=1$.

\subsection{Gibbs point processes with superstable pair interactions}
Let $u:\R^d\to\R\cup\{+\infty\}$ with $u(0)=+\infty$ 
be an even interaction pair potential, 
for which there exist $r_0>0$ and decreasing positive functions 
$\varphi:(0,r_0)\to\R^+$ and $\psi:[0,\infty)\to\R^+$ with
\[
   \int_0^{r_0} r^{d-1}\varphi(r)\dr \,=\, +\infty\,, \qquad
   \int_{r_0}^\infty r^{d-1}\psi(r)\dr \,<\, \infty\,,
\]
such that
\be{superstable}
\begin{aligned}
   u(x) &\,\geq\, \varphi(|x|)\,, \qquad 0 <|x|<r_0\,, \\[1ex]
   |u(x)| &\,\leq\, \psi(|x|)\,, \qquad |x|\geq r_0\,.
\end{aligned}
\ee
Then the Mayer $f$-function
\be{f}
   f(x) \,=\, e^{-\beta u(x)} - 1\,, \qquad x\in\R^d\,,
\ee
belongs to $L^1(\R^d)\cap L^\infty(\R^d)$ for every value $\beta>0$ of the
so-called \emph{inverse temperature}. 

For $\gamma\in\Gamma_0$ consider the structural Hamiltonian
\[
   H(\gamma) \,=\, \frac{1}{2} \sum_{x\neq y \in\gamma} u(x-y)\,, 
\]
and -- for $\gamma\in\Gamma_0$ and $\eta\in\Gamma$ -- let
\[
   W(\gamma\mid\eta) \,=\, 
   \begin{cases} 
      {\displaystyle \sum_{x\in\gamma,y\in\eta} u(x-y)}\,,\quad & 
      \text{if ${\displaystyle \sum_{x\in\gamma,y\in\eta}|u(x-y)|<\infty}$}\,, \\[2ex]
      \phantom{xxx}+\infty\,, & \text{else}\,,
   \end{cases}
\]
be the associated interaction. By some abuse of notation we will also write
$H(\xx_n)$ and $W(\xx_n\mid\eta)$ instead of $H(\gamma)$ and $W(\gamma\mid\eta)$
with $\xx_n=(x_1,\dots,x_n)$, if $\gamma=\{x_1,\dots,x_n\}$;
we proceed likewise for the second argument of $W$.
Note that $H(\xx_n)=+\infty$ if some entries $x_i$ and $x_j$ of $\xx_n$ coincide.
It is known that this Hamiltonian is superstable (cf.~\cite{Ruelle70}), 
which implies, in particular, the stability bound
\be{B}
   H(\gamma) \,\geq\, \,-\, B\,\#\gamma \qquad 
   \text{for every $\gamma\in\Gamma$}
\ee
and some stability constant $B>0$.

Associated with $\beta>0$, $u$ of \req{superstable}, and any activity $z>0$ 
-- which plays a similar role as the rate of a Poisson process --
there is at least one stationary point process $\GG$
which satisfies the so-called GNZ-equation 
(Georgii-Nguyen-Zessin~\cite{Georgii76,Zessin79})
\begin{align}\label{eq:GNZ} 
    \E\Bigl[\sum_{x\in\gamma}G(x,\gamma)\Bigr] \,=\,
    \E\left[z\int_{\R^d}  G(x,\gamma\cup \{x\})\,
            e^{ -\beta W(x\mid \gamma)}  \dx\right] \tag{GNZ}
\end{align}
for every nonnegative function $G:\R^d\times\Gamma\to[0,+\infty]$,
where we, again, have used the short-hand notation
$\xx_n=(x_1,\dots,x_n)$ with $x_i\in\R^d$.
Such a point process $\GG$ is called a $(\beta,z,u)$-\emph{Gibbs point process}.
From \req{GNZ} one can also derive a multivariate version of the GNZ-equation,
namely
\begin{align}\label{eq:MGNZ}
    \E\Biggl[\sum_{x_1\neq \dots \neq x_n\in\gamma}G(\xx_n,\gamma)\Biggr] \,=\,
    \E\left[ z^n\int_{(\R^d)^n}  G(\xx_n,\gamma\cup \{\xx_n\})
             e^{- \beta H(\xx_n)-\beta W(\xx_n\mid \gamma)}  \dxx_n\right] \tag{MGNZ}
\end{align}
for every nonnegative function $G:(\R^d)^n\times\Gamma\to[0,\infty]$.
It further follows from \req{GNZ}, \req{MGNZ} and \req{useofrho} that
$\GG$ admits correlation functions given by
\begin{subequations}
\label{eq:rho}
\be{rho-1}
   \rho^{(1)}(x) \,=\, \E\bigl[z e^{ -\beta W(x\mid \gamma)} \bigr]\,, \qquad x\in\R^d\,,
\ee
and
\be{rho-n}
   \rho^{(n)}(\xx_n) 
   \,=\, \E\bigl[z^ne^{-\beta H(\xx_n)-\beta W(\xx_n\mid\gamma)}\bigr]\,, \qquad
   \xx_n\in(\R^d)^n\,, 
\ee
\end{subequations}
for $n=2,3,\dots$, respectively. As has been established in \cite{Ruelle70}
these correlation functions satisfy a Ruelle bound for some $q=q(\beta,z,u)>0$. 
Note that if $u$ of \req{superstable} belongs to 
$L^\infty_{\rm loc}(\R^d\setminus\{0\})$ and $\beta>0$ is given
then $z>0$ can be tuned to realize any positive density $\rho$ for some 
$(\beta, z,u)$-Gibbs point process $\GG$, cf.~\cite{FrHa22}.

\section{Hyperuniformity and structure factor}
\label{Sec:Nonhyperuniformity}
For a stationary point process $\PP$ with intensity $\rho$
and $\Lambda=B_\ell(0)\subset\R^d$ the random variable
\be{N}
   N_\Lambda(\gamma) \,=\, \#(\gamma\cap\Lambda)
   \,=\, \sum_{x\in\gamma} 1_\Lambda(x)\,, \qquad \gamma\in\Gamma\,,
\ee
has the expectation $\E[N_\Lambda]=\rho|\Lambda|$. Accordingly, this is an
extensive quantity, i.e., its value is proportional to the volume of $\Lambda$.
The point process $\PP$ is called \emph{hyperuniform}, 
if the variance of $N_\Lambda$ is growing with a
smaller rate, i.e., if
\begin{align*}
    \lim_{\ell\to\infty} \frac{\Var\left[N_\Lambda\right]}{|
    \Lambda|} \,=\, 0\,.
\end{align*}
In the context of statistical physics non-hyperuniformity of the constituents
of a homogeneous fluid is measured in terms of the so-called 
\emph{compressibility} 
\be{compressibility}
   \kappa \,=\, \frac{\beta}{\rho} \,
                \lim_{\ell\to\infty}\frac{\Var\bigl[N_\Lambda]}{|\Lambda|}\,.
\ee 
\color{black}

The \emph{covariance measure} $\mathscr{C}$ of $\PP$ is defined by
\[
    \Cov \Bigl[\sum_{x \in \gamma}\testfctphi_1(x), 
               \sum_{x \in \gamma}\testfctphi_2(x)\Bigr] 
    \,=\, \rho\int_{\R^d}\int_{\R^d}\testfctphi_1(x) \testfctphi_2(x+y) 
          \rmd\mathscr{C}(y)\dx  
\]
for every bounded
$\testfctphi_1,\testfctphi_2\colon \R^d \to \R$ of compact support, 
cf., e.g., Br\'{e}maud~\cite{Bremaud20}. The covariance measure is a 
nonnegative locally finite measure on $(\R^d,\mathcal{B}(\R^d))$.
Further, there exists a nonnegative locally finite measure $\mathscr{S}$
on $(\R^d,\mathcal{B}(\R^d))$ such that 
\be{Bartlett}
    \Var\Bigl[\sum_{x \in \gamma }\testfctphi(x)\Bigr]
    \,=\, \rho \int_{\R^d} 
                  \big|\widehat{\testfctphi}(\xi)\big|^2
               \rmd \mathscr{S}(\xi)\,.
\ee
This measure is called the \emph{Bartlett spectral measure} 
or \emph{centered diffraction measure}. 
For point processes with correlation functions it follows from \req{useofrho}
that one has
\be{Varphi}
\begin{aligned}
   \Var\Bigl[\sum_{x \in \gamma}\testfctphi(x)\Bigr]
   &
   \,=\, \E\Bigl[
   \sum_{x_1\neq x_2 \in \gamma}
   \testfctphi(x_1)
   \testfctphi(x_2)
   \Bigr]
   \,+\,\E\Bigl[
   \sum_{x \in \gamma}\big(\testfctphi(x)\big)^2
   \Bigr]
   \,-\,\left(\E\Bigl[
   \sum_{x \in \gamma}\testfctphi(x)
   \Bigr]\right)^2
      \\[1ex]
   &\,=\, \int_{(\R^d)^2}
             \testfctphi(x_1)\testfctphi(x_2) \,\omega(x_1-x_2) \dxx_2
   \,+\, \int_{\R^d}\testfctphi^2(x) \rho\dx
   \,=\, \int_{\R^d}
   \int_{\R^d}
   \testfctphi(x_1)\testfctphi(x_1+x_2) \,\rho\rmd\mathscr{C}(x_2)\dx_1
\end{aligned}
\ee
for 
\begin{align*}
    \rmd\mathscr{C}
    \,=\, \frac{1}{\rho}\, \left(\omega\dx +\rho\,\delta_0\right).
\end{align*}
Accordingly, if $\omega\in L^1(\R^d)$ then the Bartlett spectral measure 
admits a density with respect to the Lebesgue measure, namely
\be{S}   
   S(\xi) \,=\, \frac{1}{\rho}\,\bigl(\omegahat(\xi)+\rho\bigr)\,.
\ee
This density $S$ is called the \emph{structure factor} or 
\emph{structure function}. 
In this case it follows from \req{Varphi} that
\[
   \frac{1}{|\Lambda|}\,\Var\bigl[N_\Lambda\bigr]
   \,=\, \frac{1}{|\Lambda|}\int_{\Lambda^2} \omega(x_1-x_2)\dxx_2 \,+\, \rho
   \,=\, \int_{\R^d} k_\Lambda(x)\omega(x)\dx \,+\, \rho
\]
with
\[
   k_\Lambda(x) \,=\, \frac{1}{|\Lambda|}\int_\Lambda 1_\Lambda(x+x')\dx'\,,
   \qquad x\in\R^d\,,
\]
and hence, compare the proof of Proposition~\ref{prop:chi4} below,
\be{S0gleich0}
   \lim_{\ell\to\infty}\frac{1}{|\Lambda|}\,\Var\bigl[N_\Lambda\bigr] 
   \,=\, \omegahat(0) \,+\, \rho \,=\, \rho\,S(0)\,.
\ee
Therefore, a stationary point process $\PP$, which admits a structure factor, 
is hyperuniform, if and only if $S(0)=0$. In the general case it follows 
from \req{Bartlett} that $\PP$ is hyperuniform, if and only if
\be{Bjoerklund24}
    \lim_{\ell\to \infty} \ell^d \mathscr{S} \bigl(B_{1/\ell}(0)\bigr) 
    \,=\, 0\,,
\ee
cf., e.g., Bj\"orklund and Hartnick~\cite{Bjoerklund24}. 

Since the truncated correlation function of the Poisson process vanishes
identically, the associated structure factor is the constant
function $S\equiv 1$, and hence, the Poisson process is non-hyperuniform. 
On the other hand, the Bartlett spectral measure of the stationary lattice
$\LL$ is given by
\begin{align*}
    \mathscr{S} \,=\, \sum_{k \in \Z^d\setminus\{0\}}\delta_k\,,
\end{align*}
cf.~\cite{Bjoerklund24},
and does not have a density with respect to the Lebesgue measure. 
Since $\mathscr{S}(B_{1/\ell}(0)) =0$ for all $\ell>1$ it follows from 
\req{Bjoerklund24} that $\LL$ is hyperuniform.

\subsection{Determinantal point processes with projection kernels}
For determinantal point processes defined by \req{rhodpp} the structure
factor exists and is readily calculated.

\begin{proposition}\label{Prop:DPPS}
For the determinantal point process $\KK$ defined by \req{rhodpp} the structure factor is given by
\begin{align}\label{eq:SfactDPP}
    S(\xi) \,=\, 1 - \bigl|(E+\xi) \cap E\bigr|.
\end{align}
In particular, $S(0)=0$, while $S(\xi) >0$ for all $\xi \neq 0$.
\end{proposition}

\begin{proof}
Using \req{rhodpp} and \req{Ursell} one finds 
\[
   \omega(x) \,=\, 1 - K^2(x) - 1 \,=\, - K^2(x)\,.
\]
Accordingly, $\omega\in L^1(\R^d)$, and from the convolution theorem, 
the properties of $K$, and the symmetry of $E$ it follows that
\begin{align*}
    \omegahat(\xi) 
    \,=\, - (\widehat{K}\ast \widehat{K})(\xi)
    \,=\, -(\mathrm{1}_E\ast \mathrm{1}_E) (\xi)
    \,=\, -\bigl|(E+\xi) \cap E\bigr|\,.
\end{align*}
This establishes \req{SfactDPP}, which shows that
$S(0)=1-|E|=0$.

Furthermore, $S(\xi)=0$, if and only if
$|(E+\xi) \cap E|= 1$, i.e., if and only if $E = E+\xi$ up to 
Lebesgue null-sets. Suppose there is some $\xi \neq 0 $ such that this 
equality holds true. Then one can define the strip
\begin{align*}
    \mathcal{T} \,:=\,
    \left\lbrace
       x \in \R^d \,:\, 0 \leq x \cdot \frac{\xi}{|\xi|} <  |\xi|
    \right\rbrace
\end{align*}
and write 
\begin{align*}
   E  \,=\, \bigcup_{k \in \mathbb{Z}} \big(E \cap (\mathcal{T}+k\xi) \big) 
   \,=\, \bigcup_{k \in \mathbb{Z}} \big((E+k\xi) \cap (\mathcal{T}+k\xi) \big) 
   \,=\, \bigcup_{k \in \mathbb{Z}} \big((E \cap \mathcal{T})+k\xi \big) .
\end{align*}
It follows that 
\begin{align*}
    |E| \,=\, \sum_{k \in \mathbb{Z}} \,\bigl|(E \cap \mathcal{T})+k\xi \bigr|
    \,=\, \sum_{k \in \mathbb{Z}} |E \cap \mathcal{T}|
\end{align*}
and thus $ |E \cap \mathcal{T}| =0$, and hence, $|E|=0$.
Since this contradicts our assumptions on $E$, no such $\xi\neq 0$ can exist. 
\end{proof}

Note that Proposition~\ref{Prop:DPPS} and \req{S0gleich0} imply
that the determinantal point process $\KK$ is hyperuniform;
compare \cite{Ghosh21}.


\subsection{Gibbs point processes with superstable pair interactions}
For Gibbs point processes $\GG$ with superstable pair interactions 
as in \req{superstable}, Ruelle~\cite{Ruelle70} has used an argument of 
Ginibre~\cite{Ginibre67} to show that the compressibility of $\GG$ 
is positive, i.e., that these point processes are non-hyperuniform.
Using the GNZ-equation this was extended to a more general class of 
Gibbs point processes by Dereudre and Flimmel~\cite{Flimmel24}.
One can further elaborate on this argument to show that the
structure factor of $\GG$ is strictly positive.

\begin{theorem}
\label{Thm:structureGibbs}
For a $(\beta,z,u)$-Gibbs point process $\GG$ with $\beta >0 $, $z>0$, and 
$u$ as in \req{superstable}, the Bartlett spectral measure 
$\mathscr{S}$ is strictly positive.
\end{theorem}
 
\begin{proof}
The idea of the proof is to use two different random variables with the same 
expectation. 
Let $\testfctphi$ be a bounded function with compact support, and consider
the random variable
\[
   \PhiG(\gamma) \,=\, \sum_{x\in\gamma} \testfctphi(x)\,, \qquad \gamma\in\Gamma\,.
\]
According to \req{GNZ}, the random variable $\PhiG'$ given by
\[
    \PhiG'(\gamma) 
    \,=\, z\int_{\R^d} \testfctphi(x)\,e^{-\beta W(x\mid\gamma)} \dx\,, 
    \qquad \gamma\in\Gamma\,,
\]
has the same expectation as $\PhiG$, namely
\[
   \E\bigl[\PhiG'\bigr] \,=\, \E\bigl[\PhiG\bigr]
   \,=\, \rho\int_{\R^d}\testfctphi(x)\dx\,,
\]
where the latter identity is a consequence of \req{useofrho}.
Furthermore, compare \req{Varphi},
\be{absNxi}
    \E \left[\PhiG^2\right]
    \,=\, \int_{(\R^d)^2} 
             \testfctphi(x_1)\testfctphi(x_2)\,\rho_2(x_1-x_2)\dxx_2
          \,+\, \rho\int_{\R^d}\testfctphi^2(x)\dx\,.
\ee

Next we compute $\E\left[\PhiG\PhiG'\right]$ by using \req{GNZ} and obtain
\begin{align*}
    \E \left[\PhiG\PhiG'\right]
    &\,=\, \E\left[
               \sum_{x_1 \in\gamma} \testfctphi(x_1)\
               z\int_{\R^d} 
                   \testfctphi(x_2)\,e^{-\beta W(x_2 \mid \gamma)} \dx_2
             \right] \\
    &\,=\, \E\left[ z^2 \int_{(\R^d)^2} \testfctphi(x_1)\testfctphi(x_2)\,  
                    e^{-\beta W(x_2 \mid \gamma\cup\{x_1\})} 
                    e^{-\beta W(x_1\mid\gamma)}\dxx_2
             \right] \\[1ex]
    &\,=\, \E \left[ z^2 \int_{(\R^d)^2} \testfctphi(x_1)\testfctphi(x_2) \, 
                     e^{-\beta H(\xx_2) -\beta W(\xx_2 \mid \gamma)} \dxx_2
              \right].
\end{align*}
It therefore follows from \req{rho-n} that
\be{NxiNxiprime}
    \E \left[\PhiG\PhiG'\right]
    \,=\, \int_{(\R^d)^2} 
             \testfctphi(x_1)\testfctphi(x_2)\,\rho_2(x_1-x_2)\dxx_2\,.
\ee
Likewise we compute
\begin{align*}
    \E \left[(\PhiG')^2\right]
    &\,=\, \E\left[ z^2 
                    \int_{(\R^d)^2}
                       \testfctphi(x_1)\testfctphi(x_2)\,
                       e^{-\beta W(\xx_2 \mid \gamma)} 
                    \dxx_2\ \right] \\[1ex]
    &\,=\, \E\left[ z^2 
                    \int_{(\R^d)^2}
                       \testfctphi(x_1)\testfctphi(x_2)\,
                       e^{-\beta H(\xx_2) - \beta W(\xx_2 \mid \gamma)} 
                    \dxx_2\ \right] 
     \\[1ex]
    &\phantom{\,=}\ \,-\,
           \E\left[ z^2
                    \int_{(\R^d)^2} 
                       \testfctphi(x_1)\testfctphi(x_2)f(x_2-x_1)\,
                       e^{-\beta W(\xx_2 \mid \gamma)} 
                    \dxx_2\ \right] \\[1ex]            
    &\,=\, \int_{(\R^d)^2} 
              \testfctphi(x_1)\testfctphi(x_2)\,\rho_2(x_1-x_2)\dxx_2
           \,-\,
           \E\left[ z^2
                    \int_{(\R^d)^2} 
                       \testfctphi(x_1)\testfctphi(x_2)f(x_2-x_1)\,
                       e^{-\beta W(\xx_2 \mid \gamma)} 
                    \dxx_2\ \right]          
\end{align*}
by utilizing the Mayer $f$-function~\req{f}. For brevity, define
\begin{align*}
    \mathcal{R}
    \,=\, z^2 \int_{(\R^d)^2} 
                 \testfctphi(x_1)\testfctphi(x_2) f(x_2-x_1) \,
                 \E\left[ e^{-\beta W(\xx_2 \mid \gamma)} \right] \dxx_2\,,
\end{align*}
so that we can rewrite
\be{absNxiprime}
    \E \left[(\PhiG')^2\right]
    \,=\, \int_{(\R^d)^2} 
             \testfctphi(x_1)\testfctphi(x_2)\,\rho_2(x_1-x_2)\dxx_2
          \,-\, \mathcal{R}\,.
\ee
As we will show in Lemma~\ref{Lem2.1}
the expectation in the definition of $\mathcal{R}$ is uniformly bounded
in $(\R^d)^2$. This implies that there is some constant $C>0$ such that
\be{Rxibound}
\begin{aligned}
    \left|\mathcal{R} \right| 
    &\,\leq\, C z^2
              \int_{(\R^d)^2} \bigl|\testfctphi(x_1)\testfctphi(x_2)f(x_1-x_2)\bigr|\dxx_2
     \,=\, C z^2
           \int_{\R^d} |\testfctphi(x)\bigr|\,(|\testfctphi|*|f|)(x) \dx\\[2ex]
    &\,\leq\, C z^2
              \norm{\testfctphi}_{L^2} \bigl\||\testfctphi|*|f|\big\|_{L^2}
     \,\leq\, C z^2\,\norm{f}_{L^1}\norm{\testfctphi}_{L^2}^2\,.              
\end{aligned}
\ee

Now let
\[
    \PhiG_\eps \,=\, (1-\eps) \PhiG \,+\, \eps \PhiG'
\]
for some $\eps>0$ to be chosen later.
Using \req{absNxi}, \req{NxiNxiprime}, and \req{absNxiprime} it follows that
\begin{align*}
   \E\left[\PhiG_\eps^2\right]
   &\,=\, \int_{(\R^d)^2} 
             \testfctphi(x_1)\testfctphi(x_2)\,\rho_2(x_1-x_2)\dxx_2
          \,+\, (1-\eps)^2\rho\,\norm{\testfctphi}_{L^2}^2
          \,-\, \eps^2\mathcal{R} \\[1ex]
   &\,=\, \E\left[\PhiG^2\right]
          \,+\,(\eps^2-2\eps)\rho\,\norm{\testfctphi}_{L^2}^2
          \,-\, \eps^2\mathcal{R}\,.
\end{align*}
Since we have
\[
   \E\bigl[\PhiG_\eps\bigr]^2 \,=\, \E\bigl[\PhiG\bigr]^2
\]
by the construction of $\PhiG'$ we conclude that
\begin{align*}
    \Var\bigl[\PhiG\bigr]
    &\,=\, \Var\bigl[\PhiG_\eps\bigr]
           \,+\, (2\eps\rho \,-\, \eps^2\rho)\norm{\testfctphi}_{L^2}^2
           \,+\, \eps^2\mathcal{R}\,.
\end{align*}
It therefore follows from \req{Rxibound} and the nonnegativity of the
variance of $\PhiG_\eps$ that
\[
   \Var\bigl[\PhiG\bigr]
   \,\geq\, \bigl(2\eps\rho \,-\, C'\eps^2)\norm{\testfctphi}_{L^2}^2
\]
for some $C'>0$, independent of $\testfctphi$. Accordingly, choosing 
$\eps=\rho/C'$ we obtain that
\[
   \Var\bigl[\PhiG\bigr] \,\geq\, (\rho^2/C')\,\norm{\testfctphi}_{L^2}^2
   \,=\, (\rho^2/C')\,\norm{\widehat\testfctphi}_{L^2}^2\,.
\]

It therefore follows from \req{Bartlett} that
\[
   \int_{\R^d} \bigl|\widehat\testfctphi(\xi)\bigr|^2\rmd\mathscr{S}(\xi)
   \,\geq\, (\rho/C')\,\norm{\widehat\testfctphi}_{L^2}^2\,,
\]
and since this result holds true for every bounded $\testfctphi$ with 
compact support, the Bartlett spectral measure is strictly positive.
\end{proof}

\begin{remark}
\label{Rem:gasphase}
\rm
Ruelle has shown, cf.~\cite[Theorem~4.4.8]{Ruel69}, 
that the truncated correlation function $\omega$ of $\GG$
belongs to $L^1(\R)$, provided that the activity satisfies
\be{z0}
   z\,<\, z_0 \,=\, \frac{1}{\norm{f}_{L^1}}\,\frac{1}{e^{2\beta B+1}}\,.
\ee
In the context of statistical physics, the corresponding range of activities is 
commonly associated with the \emph{gas phase} of the fluid under consideration.
It follows that the corresponding Gibbs point processes admit a 
continuous structure factor $S$ according to \req{S}, and this is a 
strictly positive function by virtue of Theorem~\ref{Thm:structureGibbs}.
\fin
\end{remark}

\section{Second order statistics}
\label{Sec:2ndorder}
While hyperuniformity quantifies the asymptotic variance of the number of 
points it does not say anything about the \emph{local disorder} of the points 
in $\Lambda$. We now want to investigate this property somewhat further. 
For this we look at second order functionals, 
e.g.~the number of neighbor points within a given distance $R>0$ in $\Lambda$. 
For the stationary lattice $\LL$ one expects that this number is roughly equal 
to a (deterministic) multiple of $N_\Lambda$ up to boundary corrections, 
and thus its variance should grow more slowly than $|\Lambda|$; 
see Section~\ref{Subsec:Technicalities-lattice} for a rigorous computation. 
However, for general point processes the answer is not as obvious.

For $\Lambda=B_\ell(0)$ we call $V_\Lambda$ a \emph{second order statistics},
if there is an even function $\testfct:\R^d\to\R$ with suitable properties, 
such that
\be{G}
   \Testfct_\Lambda (\gamma) 
   \,=\, \sum_{x_1 \neq x_2 \in \gamma}\chiL(x_1)\chiL(x_2) \testfct(x_1-x_2)\,.
\ee
It follows that
\begin{align*}
   \Testfct_\Lambda(\gamma)^2
   &\,=\, \!\!\!\sum_{x_1\neq \dots\neq x_4\in\gamma} \!
             \Bigl(\prod_{i=1}^4\chiL(x_i)\Bigr)  
             \testfct(x_1-x_2)\testfct(x_3-x_4)
          \,+\ 4\!\!\!\!\sum_{x_1\neq x_2 \neq x_3\in\gamma}\!
             \Bigl(\prod_{i=1}^3 \chiL(x_i)\Bigr)  
             \testfct(x_1-x_2)\testfct(x_1-x_3)\\[1ex]
   &\qquad \          \,+\ 2\!\!\!\sum_{x_1\neq x_2\in\gamma}\!\!
             \chiL(x_1)\chiL(x_2)  
             \testfct(x_1-x_2)^2\,,
\end{align*}
and if the stationary point process $\PP$ admits 
correlation functions, then we further deduce from \req{useofrho} that
\be{EV}
   \E\bigl[ \Testfct_\Lambda\bigr] 
   \,=\, \int_{\Lambda^2} 
   \testfct(x_1-x_2)\rho_2(x_1-x_2)\dxx_2
\ee
and
\be{variance2ndord}
\begin{aligned}
   \E\bigl[\Testfct_\Lambda^2\bigr]
   &\,=\, \int_{\Lambda^4} \testfct(x_1-x_2)
   \testfct(x_3-x_4)\rho^{(4)}(\xx_{4})\dxx_4\\[1ex]
   &\phantom{\,=\,}
          \ +\, 4 \int_{\Lambda^3} 
                     \testfct(x_1-x_2)
                     \testfct(x_1-x_3) \rho^{(3)}(\xx_{3})
                  \dxx_3
          \,+\, 2 \int_{\Lambda^2} 
          \testfct(x_1-x_2)^2 \rho_2(x_1-x_{2})\dxx_2\,,
\end{aligned} 
\ee
provided that the integrals converge.
Introducing
\be{chi4}
   \chi^{(4)}(\xx_4) 
   \,=\, \rho^{(4)}(\xx_4) \,-\, \rho_2(x_1-x_2)\rho_2(x_3-x_4)\,.
\ee
we can therefore rewrite
\be{VarV}
\begin{aligned}
   \Var\bigl[\Testfct_\Lambda\bigr]
   &\,=\, \int_{\Lambda^4} 
   \testfct(x_1-x_2)
   \testfct(x_3-x_4)\chi^{(4)}(\xx_4)\dxx_4
           \\[1ex]
   &\phantom{\,=\,}
          \ +\, 4 \int_{\Lambda^3}
                     \testfct(x_1-x_2)
                     \testfct(x_1-x_3)\rho^{(3)}(\xx_3)\dxx_3
          \,+\, 2\int_{\Lambda^2} 
          \testfct(x_1-x_2)^2\rho_2(x_1-x_2)\dxx_2\,.
\end{aligned}
\ee

\begin{proposition}\label{prop:chi4}
Let $\PP$ be a stationary point process with correlation functions,
and assume that $\rho^{(2)}$ and $\rho^{(3)}$ are bounded
and that $\chi^{(4)}$ of \req{chi4} satisfies  
\begin{align}
\label{eq:chi4-bound}
   \int_{\R^d}\left|\chi^{(4)}(x_1,x_2,x_3,x_4+x_3)\right|\dx_3 \in L^\infty((\R^{d})^3)
\end{align}
as a function of $(x_1,x_2,x_4)\in(\R^d)^3$.
Further, let $V_\Lambda$ be defined 
by \req{G} for some even function $v\in L^1(\R^d)\cap L^2(\R^d)$. Then
\be{EVlim}
   \lim_{\ell\to\infty} \frac{1}{|\Lambda|}\,\E\bigl[V_\Lambda\bigr]
   \,=\, \int_{\R^d} v(x)\rho_2(x)\dx
\ee
and
\be{Hank17b}
\begin{aligned}
   \lim_{\ell\to\infty}
   \frac{1}{|\Lambda|}\,\Var\bigl[\Testfct_\Lambda\bigr]
   &\,=\, \int_{\R^d} \testfct(x_1) 
             \int_{\R^d} 
                \testfct(x_2)\int_{\R^d}\chi^{(4)}(x_1,0,x_3,x_3+x_2)\dx_3\dx_2\dx_1
                \\[1ex]
   &\phantom{\,=\,}
          \ +\, 4 \int_{\R^d} \testfct(x_1) 
                     \int_{\R^d} \testfct(x_2)\rho^{(3)}(0,x_1,x_2)\dx_2\dx_1
          \ +\, 2\int_{\R^d} \testfct^2(x)\rho_2(x)\dx\,.
\end{aligned}
\ee
\end{proposition}

\begin{proof}
Using the symmetry of $\testfct$ and the translation invariance of the 
correlation functions and of $\chi^{(4)}$, it follows from 
\req{EV} and \req{VarV} that
\[
   \frac{1}{|\Lambda|}\,\E\bigl[V_\Lambda\bigr] \,=\, \int_{\R^d} k_\Lambda^{(1)}(x)v(x)\rho_2(x)\dx
\]
and
\begin{align*}
   \frac{1}{|\Lambda|}\,\Var\bigl[V_\Lambda\bigr]
   &\,=\, \int_{(\R^d)^3} 
             k_\Lambda^{(3)}(\xx_3)\,
             \testfct(x_1)\testfct(x_2)
             \chi^{(4)}(x_1,0,x_3,x_3+x_2)\dxx_3
           \\[1ex]
   &\phantom{\,=\,}
          \ +\, 4 \int_{(\R^d)^2} 
                     k_\Lambda^{(2)}(\xx_2)\,
                     \testfct(x_1)\testfct(x_2)
                     \rho^{(3)}(0,x_1,x_2)\dxx_2
          \,+\, 2 \int_{\R^d} k_\Lambda^{(1)}(x)\, \testfct(x)^2\rho_2(x)\dx
\end{align*}
with
\begin{subequations}
\begin{align}
\label{eq:k0}
   k_\Lambda^{(1)}(x) 
   &\,=\, \frac{1}{|\Lambda|}\int_\Lambda \chiL(x+x')\dx'\,, \qquad x\in\R^d\,,
   \\[1ex]
\label{eq:k2}
   k_\Lambda^{(2)}(\xx_2)
   &\,=\, \frac{1}{|\Lambda|}\int_\Lambda \chiL(x_1+x')\chiL(x_2+x')\dx'\,,
   \qquad \xx_2\in(\R^d)^2\,,\\[1ex]
\intertext{and}
\nonumber
   k_\Lambda^{(3)}(\xx_3)
   &\,=\, \frac{1}{|\Lambda|}\,\int_{\Lambda} 
          \chiL(x_1+x')\chiL(x_3+x')\chiL(x_2+x_3+x')\dx'\,, 
   \qquad \xx_3\in(\R^d)^3\,.
\end{align}
\end{subequations}
The three functions $k_\Lambda^{(1)}$, $k_\Lambda^{(2)}$ 
and $k_\Lambda^{(3)}$ are bounded by one, and converge pointwise to one as the radius $\ell$ of $\Lambda$ goes to infinity.
By the dominated convergence theorem the claim therefore follows.
\end{proof}

Replacing $v$ by $|v|$ it follows from \req{EV}
that the random variable $|V_\Lambda|$
has a finite expectation value, and hence, $V_\Lambda(\gamma)$ is finite
with probability one for $\gamma\in\Gamma$ under the assumptions on
$\PP$ and $v$ of Proposition~\ref{prop:chi4}.

According to \req{EVlim} $\bigl|\E\bigl[V_\Lambda\bigr]\bigr|$ 
is an extensive quantity generically. What can be said about the
variance of this random variable?
The easiest example to look at is the Poisson process with 
rate $\rho$. Here, $\chi^{(4)}=0$ by virtue
of \req{rhopoiss}, and hence, \req{Hank17b} implies that 
\begin{align*}
    \lim_{\ell\to\infty}
   \frac{1}{|\Lambda|}\,\Var\bigl[\Testfct_\Lambda\bigr]
    \,=\, 4 \rho^3\Bigl(\int_{\R^d}\testfct(x)\dx\Bigr)^2 \,+\, 2\rho^2\,\norm{\testfct}^2_{L^2}
    \,>\, 0
\end{align*}
for every even function $\testfct\in L^1(\R^d)\cap L^2(\R^d)\setminus\{0\}$.

In the sequel we investigate
this question for the hyperuniform determinantal point process $\KK$ and the
non-hyperuniform Gibbs process $\GG$.

\subsection{Determinantal point processes with projection kernels}
For determinantal point processes we have the following result.

\begin{theorem}\label{Thm:NHUDPP}
Let $V_\Lambda$ be given by \req{G}. Then for the determinantal point process $\KK$ 
defined in \req{rhodpp} there holds
\be{Thm:NHUDPP}
\lim_{\ell\to \infty}\frac{1}{|\Lambda|}\Var \left[
    \Testfct_\Lambda
    \right] \,>\, 0
\ee
for every even function $\testfct\in L^1(\R^d)\cap L^2(\R^d) \setminus \{0\}$.
\end{theorem}

\begin{proof}
We employ Lemma~\ref{Lem:NHUDPP}. It states that for the determinantal 
point process $\KK$ and an even function 
$v\in L^1(\R^d)\cap L^2(\R^d)$ there holds
\be{2VarDPP-Kopie}
    \lim_{l\to \infty}\frac{1}{|\Lambda|}\Var \left[\Testfct_\Lambda\right]
    \,=\, 2\int_{\R^d} 
              \bigl(\Ghat(\xi)\bigr)^2\,\bigl|(E+\xi)\setminus E)\bigr|^2\dxi
    \,-\,2\int_{(\R^d)^2} 
              \Ghat(\xi_1)\Ghat(\xi_2)\,\theta(\xi_1,\xi_2)\rmd(\xi_1,\xi_2)
\ee
with
\[
   \theta(\xi_1,\xi_2)
   \,=\, \Bigl|\,\bigl((E+\xi_1)\setminus E\bigr)\cap
                 \bigl((E+\xi_2)\setminus(E+\xi_1+\xi_2)\bigr)\,\Bigr|\,.
\]
Take note that $\theta$ is a bounded continuous function which is symmetric in its 
two arguments. Furthermore, since
\[
   \int_{\R^d} M \cap (N+\xi)\dxi \,=\, |M|\,|N|
\]
for any two Lebesgue measurable sets $M,N\subset\R^d$, it follows that
\be{intM}
\begin{aligned}
   \int_{\R^d} \theta(\xi_1,\xi_2)\dxi_2
   &\,=\, \int_{\R^d} \Bigl|\,\bigl((E+\xi_1)\setminus E\bigr)\cap(E+\xi_2)\,
                      \Bigr|\dxi_2
          \,-\,
          \int_{\R^d} 
             \Bigl|\,\bigl(E+\xi_1)\setminus E)\bigr)
                     \cap(E+\xi_2)\cap(E+\xi_1+\xi_2)\,\Bigr|
          \dxi_2\\[1ex]
   &\,=\, \bigl|(E+\xi_1)\setminus E\bigr| 
          \,-\, \bigl|(E+\xi_1)\setminus E\bigr|\,
                \bigl| E\cap(E+\xi_1)\bigr|
    \,=\, \bigl|(E+\xi_1)\setminus E\bigr|^2\,,
\end{aligned}
\ee
because $|E|=1$.
Accordingly, the double integral ${\cal I}$ in \req{2VarDPP-Kopie}
can be bounded by using the Cauchy-Schwarz inequality and the symmetry of 
$\theta$, namely
\[
\begin{aligned}
   {\cal I} 
   &\,\leq\, \left(
                \int_{(\R^d)^2} 
                   \bigl(\Ghat(\xi_1)\bigr)^2 \,\theta(\xi_1,\xi_2) 
                   \rmd(\xi_1,\xi_2)
                \int_{(\R^d)^2} 
                   \bigl(\Ghat(\xi_2)\bigr)^2 \,\theta(\xi_1,\xi_2) 
                   \rmd(\xi_1,\xi_2)
             \right)^{1/2}
    =\, \int_{(\R^d)^2} 
             \bigl(\Ghat(\xi_1)\bigr)^2 \,\theta(\xi_1,\xi_2) \rmd(\xi_1,\xi_2)
    \\[1ex]
   &\,=\, \int_{\R^d} 
             \bigl(\Ghat(\xi_1)\bigr)^2 \int_{\R^d} \,\theta(\xi_1,\xi_2)\dxi_2
          \dxi_1
    \,=\, \int_{\R^d} 
             \bigl(\Ghat(\xi)\bigr)^2\,\bigl|(E+\xi)\setminus E\bigr|^2
          \dxi\,.
\end{aligned}
\]
Inserting this estimate into \req{2VarDPP-Kopie} we thus have
shown that
\[
   \lim_{\ell\to \infty}
   \frac{1}{|\Lambda|}\Var \left[\Testfct_\Lambda\right] \,>\, 0\,,
\]
unless there exist $a,b\in\R$ with $a^2+b^2>0$, such that
\be{ae-pointwise-identity}
   a\, \Ghat(\xi_1) \,\theta(\xi_1,\xi_2)^{1/2} 
   \,=\, b\, \Ghat(\xi_2) \,\theta(\xi_1,\xi_2)^{1/2} \qquad 
   \text{a.e. in $(\R^d)^2$\,.}
\ee
So let us assume that \req{ae-pointwise-identity} holds true. 
Due to the symmetry of $\theta$ we can assume without loss
of generality that $a\neq 0$, and since $\Ghat$ and $\theta$ 
are continuous functions, we conclude that
\be{pointwise-identity}
   \Ghat(\xi_1) \,\theta(\xi_1,\xi_2)^{1/2} 
   \,=\, \frac{b}{a}\, \Ghat(\xi_2) \,\theta(\xi_1,\xi_2)^{1/2} \qquad 
   \text{for all $\xi_1,\xi_2\in \R^d$\,.}
\ee
As shown in the proof of Proposition~\ref{Prop:DPPS}, the right-hand side
of \req{intM} is positive for every $\xi_1\neq 0$. This implies that
for every choice of $\xi_1\neq 0$ there is some 
$\xi_2=\xi_2(\xi_1)\in\R^d\setminus\{0\}$, such that 
$\theta(\xi_1,\xi_2)>0$. Moreover, by continuity there exists $\eps>0$,
such that
\[
   \theta\bigl(\xi_1',\xi_2(\xi_1)\bigr) \,>\, 0  \qquad
   \text{for \ $|\xi_1'-\xi_1|<\eps\,.$}
\]
It therefore follows from \req{pointwise-identity} that
\[
   \Ghat(\xi_1') \,=\, \frac{b}{a}\,\Ghat(\xi_2(\xi_1)) \qquad
   \text{for all $\xi_1'$ with $|\xi_1'-\xi_1|<\eps$}\,,
\]
and this implies that $\Ghat$ is constant in a neighborhood of $\xi_1$.
This being valid for all $\xi_1\neq 0$, $\Ghat$ must be constant
in all of $\R^d$, and since $\Ghat\in C_0$, we necessarily have
$\Ghat=0$. We thus have established \req{Thm:NHUDPP}.
\end{proof}

\subsection{Gibbs point processes with superstable pair interactions}
\label{Subsec:2ndorderGibbs}
For a Gibbs point process $\GG$ with a superstable pair potential we introduce,
similar to Section~\ref{Sec:Nonhyperuniformity}, 
two random variables which have the same expectation as $V_\Lambda$ of \req{G}.
To deduce the first one we rewrite
\[
   V_\Lambda(\gamma)
   \,=\, \sum_{x_2\in\gamma} G(x_2,\gamma) \qquad \text{with} \qquad
   G(x_2,\gamma) \,=\, 
   \sum_{x_1\in\gamma\setminus\{x_2\}}\chiL(\xx_2)v(x_1-x_2)\,,
\]
and utilize \req{GNZ}; this gives
\begin{align*}
   \E\bigl[V_\Lambda\bigr]
   &\,=\, \E\left[z\int_\Lambda\sum_{x_1\in(\gamma\cup\{x_2\})\setminus\{x_2\}} 
                     \chiL(x_1)v(x_1-x_2)e^{-\beta W(x_2\mid\gamma)}\dx_2\right]
                     \\[1ex]
   &\,=\, \E\left[z\int_\Lambda\sum_{x_1\in\gamma\setminus\{x_2\}} 
                     \chiL(x_1)v(x_1-x_2)e^{-\beta W(x_2\mid\gamma)}\dx_2
            \right].
\end{align*}
Concerning the sum in the final expression it does not matter whether $x_2$
is eliminated from $\gamma$ or not, because when $x_2\in\gamma$ then the
interaction term $e^{-\beta W(x_2\mid\gamma)}$ vanishes anyway. Accordingly,
the expected values of $V_\Lambda$ and 
\be{2ndorderintegral2}
   V'_\Lambda(\gamma) \,=\, 
   \sum_{x_1  \in \gamma }\chiL(x_1)\,z 
   \int_\Lambda v(x_1-x_2) e^{-\beta W(x_2 \mid \gamma)} \dx_2
\ee
are the same. On the other hand, we can utilize \req{MGNZ}, which gives
\[
   \E\bigl[V_\Lambda\bigr] \,=\, \E\bigl[V_\Lambda''\bigr]
\]
for
\be{2ndorderintegral3}
   V''_\Lambda(\gamma) \,=\, 
    \int_{\Lambda^2} 
       v(x_1-x_2)\,z^2 e^{-\beta H(\xx_2) - \beta W(\xx_2 \mid \gamma) } 
    \dxx_2\,.
\ee

\begin{lemma}
\label{Lem:Vtilde}
Let $\GG$ be a $(\beta,z,u)$-Gibbs point process with a superstable 
pair potential $u$. Further, for $\eps\geq 0$ and 
$V_\Lambda$, $V_\Lambda'$, and $V_\Lambda''$ defined as above, let
\be{Veps}
   V^\eps_\Lambda 
   \,=\, (1-\eps) V_\Lambda \,+\, 2\eps V_\Lambda' \,-\, \eps V_\Lambda''\,.
\ee
Then there holds
\be{VarVeps}
   \Var\bigl[V^\eps_\Lambda\bigr] 
   \,=\, \Var\bigl[V_\Lambda\bigr] 
         \,-\, 4\eps \int_{\Lambda^2} v(x_1-x_2)^2\rho_2(x_1-x_2)\dxx_2
         \,+\, \eps^2\mathcal{R}_v
\ee
where
\be{Rvbound}
   \frac{\mathcal{R}_v}{|\Lambda|} \,\leq\, C\,\norm{v}_{L^2}^2 
\ee
for some constant $C>0$, which only depends on $u$, $z$, and $\beta$.
\end{lemma}

We postpone the proof of Lemma~\ref{Lem:Vtilde} to 
Section~\ref{subsec:Gibbs-technical}. 

\begin{theorem}
\label{Thm1}
For every $(\beta,z,u)$-Gibbs point process $\GG$ with a superstable pair potential $u$ 
there exists a constant $c>0$, such that $V_\Lambda$ of \req{G} satisfies
\be{Thm1}
   \liminf_{\ell\to\infty}\frac{\Var\bigl[V_\Lambda\bigr]}{|\Lambda|}
   \,\geq\,
   \frac{c}{\norm{v}_{L^2}^2} \left(\int_{\R^d} v(x)^2\rho_2(x)\dx\right)^2  
\ee
for every even function $v\in L^1(\R^d)\cap L^2(\R^d)\setminus\{0\}$. 
\end{theorem}

\begin{proof}
Let $v\in L^1(\R^d)\cap L^2(\R^d)\setminus\{0\}$ be an even function, and define
$V_\Lambda$ and $V^\eps_\Lambda$ as in \req{G} and \req{Veps}, respectively. 
Then it follows from Lemma~\ref{Lem:Vtilde} that, for $\eps\geq 0$,
\begin{align*}
   \frac{1}{|\Lambda|}\,\Var\bigl[V_\Lambda\bigr] 
   &\,=\, \frac{1}{|\Lambda|}\,\Var\bigl[V^\eps_\Lambda\bigr] 
          \,+\, 4\eps\,\frac{1}{|\Lambda|}
                \int_{\Lambda^2} v(x_1-x_2)^2\rho_2(x_1-x_2)\dxx_2
          \,-\, \eps^2\,\frac{\mathcal{R}_v}{|\Lambda|} \\[1ex]
   &\,\geq\, 4\eps\,\frac{1}{|\Lambda|}
                    \int_{\Lambda^2} v(x_1-x_2)^2\rho_2(x_1-x_2)\dxx_2
             \,-\, \eps^2\,C \norm{v}_{L^2}^2 
\end{align*}
due to the nonnegativity of the variance. The right-hand side becomes
maximal for
\[
   \eps \,=\, \frac{2}{C \norm{v}_{L^2}^2}\,
     \frac{1}{|\Lambda|}\int_{\Lambda^2} v(x_1-x_2)^2\rho_2(x_1-x_2)\dxx_2\,,
\]
which gives
\[
   \frac{1}{|\Lambda|}\,\Var\bigl[V_\Lambda\bigr] 
   \,\geq\, \frac{4}{C\norm{v}_{L^2}^2}
            \left(\frac{1}{|\Lambda|}
                     \int_{\Lambda^2} v(x_1-x_2)^2\rho_2(x_1-x_2)\dxx_2
            \right)^2.
\]
Since $v^2\rho_2$ belongs to $L^1(\R^d)$,
the assertion~\req{Thm1} now follows with $c=4/C$ by taking 
$\ell\to\infty$.
\end{proof}

\begin{remark}
\label{Rem:hardcore}
\rm 
Note that the lower bound in \req{Thm1} is not necessarily positive. 
For example, if the pair potential $u$ takes the value $+\infty$ on 
$B_{r_{\rm hc}}(0)$ for some hard core radius $r_{\rm hc}>0$ then 
$\rho_2$ vanishes on $B_{r_{\rm hc}}(0)$ by virtue of \req{rho-n}, 
and hence, the right-hand side of \req{Thm1} is equal to zero 
for all $v$ supported on $B_{r_{\rm hc}}(0)$. 
In fact, in this situation $\rho^{(n)}(\xx_n)= 0$,
whenever two entries $x_i$ and $x_j$ of $\xx_n$ get as close as 
$r_{\rm hc}$; by virtue of \req{variance2ndord} this implies
that $\E(V_\Lambda^2)= 0$, too,
i.e., that the random variable $V_\Lambda$ is zero almost surely.
\fin
\end{remark}

\subsection{Conclusion}
In summary we have seen that for hyperuniform systems it is not clear
a priori whether the system is hyperuniform for second order statistics:  
For the stationary lattice the second order statistics are hyperuniform,
for the determinantal point processes $\KK$ of \req{rhodpp} they are not.

Some of these determinantal point processes are known to be 
\emph{number rigid}, cf.~\cite{Ghosh21}, 
which means that the number of points of a configuration $\gamma$ in $\Lambda$ 
can be inferred from observing $\gamma \cap (\R^d\setminus\Lambda)$. 
Since the second order statistics fail to be hyperuniform in these cases, 
nothing more can be deduced about $\gamma\cap\Lambda$; 
compare, e.g., Dereudre et al~\cite{Dereudre20}.
In contrast, for the stationary lattice the observation of 
$\gamma \cap (\R^d\setminus\Lambda)$ completely determines $\gamma\cap\Lambda$; 
this property is known as \emph{maximal rigidity} 
(cf.~Ghosh and Lebowitz~\cite{Ghosh18}).

Further, our results indicate that for general non-hyperuniform point processes 
the variance of second order statistics will be positive generically.






\section{The inverse Henderson problem}
\label{Sec:Jacobian}
As mentioned in the introduction the inverse Henderson problem is concerned
with the reconstruction of the pair potential $u$ of a $(\beta,z,u)$-Gibbs
measure $\GG$ from the corresponding radial distribution function, i.e., from $\rho_2/\rho^2$.
In computational physics the so-called \emph{inverse Monte-Carlo method}
is often employed for this purpose. This method, originally developed by
Lyubartsev and Laaksonen~\cite{LyLa95}, is an instance of the Newton iteration 
to solve the operator equation
\be{Henderson}
   F[u] \,=\, \rho_2\,,
\ee
where $F=F_{\beta,\rho}$ maps the pair potential $u$
onto the associated pair correlation function under given conditions
on temperature and density, i.e., for fixed values of $\beta$ and $\rho$.
An alternative way of interpreting this method is via the 
minimization of a (strictly convex) relative entropy type functional 
${\cal E}={\cal E}_{\beta,\rho,\rho_2}[u]$, 
of which $\nabla_u{\cal E}=(\rho_2-F[u])/2$ is the associated gradient;
see Murtola, Karttunen, and Vattulainen~\cite{MKV09} or \cite{FrHa22}.

For a given superstable pair potential $u_0$ denote by $\psi_0$ the
majorant of \req{superstable}. Without loss of generality let us
assume that $\psi_0$ is bounded in $[0,\infty)$, and 
define the Banach space $\V$ as the 
set of all even functions $v:\R^d\to\R$, for which the associated norm
\[
   \norm{v}_\V \,=\, 
   \esssup_{x\in \R^d}\,\frac{|v(x)|}{\psi_0(|x|)}
\]
is finite. This definition implies that $u_0+v$ is again a
superstable pair potential for every $v\in\V$ with sufficiently 
small norm $\norm{v}_\V$, hence we consider $\V$ to be the space of 
admissible \emph{perturbations} of $u_0$. Also take note that 
$\V\subset L^1(\R^d)\cap L^2(\R^d)$.
It can be shown, cf.~\cite{FrHa22}, 
that within the gas phase, i.e. for a given density $\rho$ which corresponds 
to an activity $z$ satisfying \req{z0},
there exists a bounded linear operator
$F'[u_0] \in \L(\V,L^\infty(\R^d))\cap\L(\V,L^1(\R^d))$ with 
\[
   \frac{\norm{F[u_0+v]-F[u_0]-F'[u_0]v}_\Y}{\norm{v}_\V} \,\to\, 0 \qquad
   \text{as $\norm{v}_\V\!\to 0$}\,,
\]
where $\Y$ may stand for $L^1(\R^d)$ and/or $L^\infty(\R^d)$.
In other words, the operator $F$ is differentiable at $u=u_0$,
and $-F'[u]$ is the Hessian of the 
relative entropy functional ${\cal E}$.
Accordingly, the inverse Monte-Carlo method amounts to solving
\[
   F'[u_0]v \,=\, \rho_2 - F[u_0]
\]
for $v\in\V$,
and to take $u=u_0+v$ as an improved approximation of the true pair potential.

It has been argued in \cite{LyLa95} that the Jacobian $F'[u_0]$ is connected 
to the covariance of the observable $V_\Lambda$ of \req{G}, when doing
Monte-Carlo simulations in the canonical ensemble for a finite volume. 
This link allows to assemble the Jacobian $F'[u_0]$ on the fly, 
when evaluating $F[u_0]$ numerically.
As pointed out in \cite[Remark~6.1]{FrHa22} this connection has to be
augmented in the context of Gibbs point processes by an additional term
which takes into account that the density is being prescribed in the inverse
Henderson problem. The corresponding modification is as follows.

\begin{proposition}
\label{Prop:Henderson}
Let $\GG$ be a $(\beta,z,u_0)$-Gibbs point process associated with a 
superstable pair potential $u_0$ and an activity $z$ in the gas phase and
corresponding density $\rho$.
Then the derivative $F'[u_0]$ of the operator $F$ of \req{Henderson} satisfies
\[
   \int_{\R^d} v(x)\bigl(F'[u_0]v\bigr)(x)\dx
   \,=\, -\frac{\beta}{2}\,
         \lim_{\ell\to\infty} 
            \frac{\Var\bigl[V_\Lambda\bigr]\Var\bigl[N_\Lambda\bigr]
                  \,-\, \Cov\bigl[V_\Lambda,N_\Lambda\bigr]^2}
                 {|\Lambda|^2}
         \ \Big/ \ 
         \lim_{\ell\to\infty} \frac{\Var\bigl[N_\Lambda\bigr]}{|\Lambda|},
\]
where $N_\Lambda$ is as in \req{N} and $V_\Lambda$ is defined by \req{G} 
for the given $v\in\V$. 
\end{proposition}

\begin{proof}
Let $v\in\V$ be given. 
Then we quote from the proof of Lemma~7.2 and Lemma~A.2 of \cite{FrHa22} that
\be{Prop:Henderson-1}
\begin{aligned}
   \int_{\R^d} v(x)\bigl(F'[u_0]v\bigr)(x)\dx
   &\,=\, -\beta \int_{\R^d}v^2(x)\rho_2(x)\dx
          \,-\, 2\beta\int_{(\R^d)^2}v(x_1)\rho^{(3)}(0,x_1,x_2)v(x_2)\dxx_2\\[1ex]
   &\phantom{\,=\ }
    \,-\,\frac{\beta}{2}\int_{(\R^d)^3} v(x_1)v(x_2)\chi^{(4)}(x_1,0,x_3,x_2+x_3)\dxx_3
    \,+\, \frac{2}{\rho\beta S(0)} 
          \left(\int_{\R^d} (\nabla_u\rho)(x) v(x)\dx\right)^2,
\end{aligned}
\ee
where $\chi^{(4)}$ is given in \req{chi4} and satisfies the 
assumption~\req{chi4-bound} of Proposition \ref{prop:chi4} for $z$ in the 
gas phase, and $\beta S(0)$ is the compressibility of the system,
cf.~\req{compressibility} and \req{S0gleich0}.
Furthermore, $\nabla_u\rho\in L^\infty(\R^d)$ denotes the gradient of $\rho$
with respect to $u$ (\cite{Hank17b,FrHa22}), given by
\be{nablaurho}
   (\nabla_u\rho) (x) 
   \,=\, -\beta \rho_2(x) 
         \,-\, \frac{\beta}{2}\int_{\R^d} \chi^{(3)}(0,x+x',x')\dx'\,,
\ee
where we have made use of the function
\be{chi3}
   \chi^{(3)}(\xx_3) \,=\, \rho^{(3)}(\xx_3) \,-\, \rho\,\rho_2(x_2-x_3)
\ee
employed in \cite{FrHa22} and \cite{Hank17b}.

On the other hand, we conclude from \req{N} and \req{G} that
\begin{align*}
   \E\bigl[V_\Lambda N_\Lambda\bigr]
   &\,=\, \E\Bigl[
               \sum_{x_1\in\gamma}\sum_{x_2\neq x_3\in\gamma}
                  \chiL(\xx_3)v(x_2-x_3)
            \Bigr] \\[1ex]
   &\,=\, \E\Bigl[
                \sum_{x_1\neq x_2\neq x_3\in\gamma} 
                  \chiL(\xx_3)v(x_2-x_3)
            \Bigr]
          \,+\, 2\,\E\Bigl[\sum_{x_1\neq x_2\in\gamma} \chiL(\xx_2)v(x_1-x_2)
                     \Bigr]\,,
\end{align*}
so that
\be{VN}
   \E\bigl[V_\Lambda N_\Lambda\bigr]
   \,=\, \int_{\Lambda^3} v(x_2-x_3)\rho^{(3)}(\xx_3)\dxx_3
         \,+\, 2\int_{\Lambda^2} v(x_1-x_2)\rho_2(x_1-x_2)\dxx_2
\ee
by virtue of \req{useofrho}. Using \req{EV} this yields
\begin{align*}
   \Cov\bigl[V_\Lambda,N_\Lambda\bigr]
   &\,=\, \int_{\Lambda^3} v(x_2-x_3)\rho^{(3)}(\xx_3)\dxx_3
          \,+\, 2\!\int_{\Lambda^2} v(x_1-x_2)\rho_2(x_1-x_{2})\dxx_2 
          \\
    &\phantom{\,=}\ \,-\, \rho\,|\Lambda|
                \int_{\Lambda^2} v(x_2-x_3)\rho_2(x_2-x_3)\rmd(x_2,x_3) \\[1ex]
   &\,=\, 2 \int_{\Lambda^2} v(x_1-x_2)\rho_2(x_1-x_{2})\dxx_2 
          \,+\, \int_{\Lambda^3}\! 
                  v(x_2-x_3)\,\chi^{(3)}(\xx_3)
               \dxx_3
\end{align*}
with $\chi^{(3)}$ of \req{chi3}. Therefore
\begin{align*}
   \frac{1}{|\Lambda|}\,\Cov\bigl[V_\Lambda,N_\Lambda\bigr]
   &\,=\, 2\int_{\R^d} k_\Lambda^{(1)}(x)\,v(x)\rho_2(x)\dx
          \,+\, \int_{(\R^d)^2} k_\Lambda^{(2)}(x_2+x_3,x_3)\,
                  v(x_2)\,\chi^{(3)}(0,x_2+x_3,x_3)
                \rmd(x_2,x_3)
\end{align*}
with $k_\Lambda^{(1)}$ of \req{k0} and $k_\Lambda^{(2)}$ of \req{k2}, 
and it follows from the dominated convergence theorem that
\begin{align*}
   \lim_{\ell\to\infty} 
      \frac{1}{|\Lambda|}\,\Cov\bigl[V_\Lambda,N_\Lambda\bigr]
   &\,=\, 2\int_{\R^d} v(x)\rho_2(x)\dx
          \,+\, \int_{\R^d} v(x)\int_{\R^d}\chi^{(3)}(0,x+x',x')\dx'\dx
    \,=\,-\frac{2}{\beta}\,\int_{\R^d} v(x)\nabla_u\rho(x)\dx\,,
\end{align*}
cf.~\req{nablaurho}. 
Together with \req{Prop:Henderson-1}, \req{Hank17b}, and \req{S0gleich0}
this establishes the assertion.
\end{proof}

We are now going to use the technique from Section~\ref{Subsec:2ndorderGibbs}
to investigate the injectivity of $F'[u_0]$, i.e., the strict convexity of the
quadratic Taylor approximation of the relative entropy functional ${\cal E}$
at $u_0$. We start with the following 
technical lemma, whose proof is postponed to 
Section~\ref{subsec:Gibbs-technical}.

\begin{lemma}
\label{Lem:VtildeN}
Let $\GG$ be a $(\beta,z,u)$-Gibbs point process with a superstable pair potential $u$
and $\eps>0$. Then the covariances of the random variables $V_\Lambda$, 
$V^\eps_\Lambda$, and $N_\Lambda$ of \req{G}, \req{Veps}, and \req{N}, 
respectively, satisfy
\be{CovVepsN}
   \Cov\bigl[V^\eps_\Lambda,N_\Lambda\bigr] 
   \,=\, \Cov\bigl[V_\Lambda,N_\Lambda\bigr]\,. 
\ee
\end{lemma}

Now we are in position to prove the following result.

\begin{theorem}
\label{Thm2}
Let $\GG$ be a $(\beta,z,u_0)$-Gibbs point process associated with a superstable
pair potential $u_0$ and an activity $z$ in the gas phase with corresponding
density $\rho$.
Then there exists a constant $c'>0$, such that
\be{Thm2}
   \int_{\R^d}v(x)\bigl(F'[u_0]v\bigr)(x)\dx \,\leq\, 
   -\frac{c'}{\norm{v}_{L^2}^2} \left(\int_{\R^d} v(x)^2\rho_2(x)\dx\right)^2  
\ee
for every $v\in\V\setminus\{0\}$. Accordingly, if $\rho_2(x)>0$ for every
$x\in\R^d\setminus\{0\}$ then $F'[u_0]$ is injective on $\V$.
\end{theorem}

\begin{proof}
Let $v\in\V\setminus\{0\}$, $\Lambda=B_\ell(0)\subset\R^d$ for some $\ell>0$, 
and $\eps>0$. Then it follows from Lemma~\ref{Lem:Vtilde} and 
Lemma~\ref{Lem:VtildeN} that
\begin{align*}
   &\Var\bigl[V_\Lambda\bigr]\Var\bigl[N_\Lambda\bigr]
    \,-\, \Cov\bigl[V_\Lambda,N_\Lambda\bigr]^2\\[1ex]
   &\qquad 
    \,=\, \Var\bigl[V_\Lambda^\eps\bigr]\Var\bigl[N_\Lambda\bigr]
          \,-\, \Cov\bigl[V^\eps_\Lambda,N_\Lambda\bigr]^2
          \,+\, \Bigl(4\eps \int_{\Lambda^2} v(x_1-x_2)^2\rho_2(x_1-x_2)\dxx_2
                      \,-\, \eps^2\mathcal{R}_v
                \Bigr)\Var\bigl[N_\Lambda\bigr]\\[1ex]
   &\qquad 
    \,\geq\, \Bigl(4\eps \int_{\Lambda^2} v(x_1-x_2)^2\rho_2(x_1-x_2)\dxx_2
                \,-\, \eps^2\mathcal{R}_v
          \Bigr)\Var\bigl[N_\Lambda\bigr]\,.
\end{align*}
The right-hand side can be estimated as in the proof of Theorem~\ref{Thm1}
to obtain
\[
   \frac{1}{|\Lambda|^2}
   \Bigl(\Var\bigl[V_\Lambda\bigr]\Var\bigl[N_\Lambda\bigr]
         \,-\, \Cov\bigl[V_\Lambda,N_\Lambda\bigr]^2\Bigr)
   \,\geq\, \frac{4}{C\norm{v}_{L^2}^2}
            \left(\frac{1}{|\Lambda|}
                     \int_{\Lambda^2} v(x_1-x_2)^2\rho_2(x_1-x_2)\dxx_2
            \right)^2
            \frac{\Var\bigl[N_\Lambda\bigr]}{|\Lambda|}\,.
\]
Therefore \req{Thm2} follows from Proposition~\ref{Prop:Henderson} 
with $c'=2\beta/C$ by letting $\ell\to\infty$.
\end{proof}


\begin{remark}
\rm
It has further been shown in \cite{FrHa22} that if $u_0$ is a
\emph{Lennard-Jones type pair potential}, i.e., when $\varphi$ and $\psi$
of \req{superstable} are given by
\[
   \varphi(r) \,=\, cr^{-\alpha} \qquad \text{and} \qquad
   \psi(r)\,=\, C(1+r^2)^{-\alpha/2} 
\]
for some $\alpha>d$ and suitable $c,C>0$, then the Jacobian $F'[u_0]$ 
has an extension to a selfadjoint negative semidefinite operator 
in $\L(L^2(\R^d))$. 
For the same class of pair potentials it is also known that 
$\rho_2(x)>0$ for $x\neq 0$, provided that the density is sufficiently
small; see \cite[Proposition~3.1]{Hank18}.
It therefore readily follows from Theorem~\ref{Thm2} by continuity
that the corresponding extension $F'[u_0]\in\L(L^2(\R^d))$ is, in fact, negative, 
i.e., its null space is trivial.
\fin
\end{remark}

\section{Auxiliary results}
\label{Sec:Auxiliaries}

\subsection{The stationary lattice}
\label{Subsec:Technicalities-lattice}
For second order statistics of the stationary lattice $\LL$ a natural condition 
on the probing function $v$ of the random variable $V_\Lambda$ of \req{G} is that
\[
   \sum_{x\in\Z^d\setminus\{0\}} |v(x)| \,<\, \infty\,.
\]
Under this condition the expected value of $V_\Lambda$ satisfies
\begin{align}
\nonumber
   \frac{\E\bigl[V_\Lambda\bigr]}{|\Lambda|}
   &\,=\, \frac{1}{|\Lambda|}
          \int_{[0,1]^d} \sum_{k_1\neq k_2\in\Z^d}
                     1_\Lambda(k_1+x)1_\Lambda(k_2+x)\testfct(k_1-k_2)\dx\\[1ex]
\nonumber
   &\,=\, \frac{1}{|\Lambda|}
          \sum_{j\in\Z^d\setminus\{0\}} \testfct(j)
          \sum_{k\in\Z^d} \int_{[0,1]^d} 
                     1_\Lambda(k+x)1_\Lambda(j+k+x)\dx\\[1ex]
\label{eq:VarV-lattice}
   &\,=\, \sum_{j\in\Z^d\setminus\{0\}} \testfct(j)\,
          \frac{1}{|\Lambda|}\int_{\R^d} 1_\Lambda(x)1_\Lambda(j+x)\dx
    \,=\, \sum_{j\in\Z^d\setminus\{0\}} \testfct(j)\,
          \frac{\bigl|\Lambda\cap(\Lambda-j)\bigr|}{|\Lambda|}\,.
\end{align}
Since the fraction in the final term is bounded by one and converges to
one as the radius $\ell$ of $\Lambda=B_\ell(0)$ goes to infinity, we see that
\[
   \lim_{\ell\to\infty} \frac{\E\bigl[V_\Lambda\bigr]}{|\Lambda|}
   \,\to\!\! \sum_{j\in\Z^d\setminus\{0\}} \testfct(j)\,.
\]
\color{black}

In two space dimensions we further compute
\begin{align}
\nonumber
    \E\bigl[V_\Lambda^2\bigr] 
    &\,=\, \E\Biggl[\Bigl(
    \sum_{x_1\neq x_2\in\gamma}
    1_{\Lambda}(x_1)1_{\Lambda}(x_2) \,
    \testfct(x_1-x_2)
    \Bigr)^2
    \Biggr]
    \,=\,\int_{[0,1]^2}
    \Bigl(
    \sum_{j,k\in \Z^2\atop j\neq 0}  
    1_{\Lambda}(j+k+x)1_{\Lambda}(k+x) \,
    \testfct(j)
    \Bigr)^2
    \rmd x \\[1ex]
\nonumber
    &\,=\, \int_{[0,1]^2}
    \sum_{j_1,j_2\in\Z^2\setminus\{0\}} \testfct(j_1)\testfct(j_2)
    \sum_{k_1,k_2\in\Z^2} 
    1_{\Lambda}(j_1+k_1+x)1_{\Lambda}(k_1+x) 
    1_{\Lambda}(j_2+k_1+k_2+x)1_{\Lambda}(k_1+k_2+x) \dx \\[1ex]
\nonumber
    &\,=\, 
    \sum_{j_1,j_2\in\Z^2\setminus\{0\}} 
    \testfct(j_1)\testfct(j_2)
    \sum_{k\in \Z^2} \int_{\R^2}  
    1_{\Lambda}(j_1+x)1_\Lambda(x)1_\Lambda(j_2+k+x)1_\Lambda(k+x)
    \dx\\[1ex]
\nonumber
    &\,=\, 
    \sum_{j_1,j_2\in\Z^2\setminus\{0\}} 
    \testfct(j_1)\testfct(j_2)
    \sum_{k\in \Z^2}
    \Bigl|\Lambda\cap(\Lambda-j_1)\cap(\Lambda-k)
          \cap(\Lambda-k-j_2)\Bigr|\\[1ex]
\label{eq:techlattice1}
    &\,=\,
     \ell^2 \!\!\! 
     \sum_{j_1,j_2\in\Z^2\setminus\{0\}} \testfct(j_1)\testfct(j_2)
    \sum_{k\in \Z^2} g_{j_1,j_2}(k)\,,
\end{align}
where
\[
   g_{j_1,j_2}(k) \,=\, \Bigl|\Omega_{j_1}\cap(\Omega_{j_2}-k/\ell)\Bigr|\,, \qquad 
   k\in\R^2\,, \ j_1,j_2\in\Z^2\setminus\{0\}\,,
\]
with
\[
   \Omega_j \,=\, \Bigl(B_1(0)+\frac{j}{2\ell}\Bigr)
             \cap \Bigl(B_1(0)-\frac{j}{2\ell}\Bigr) 
   \qquad \text{for $j\in\Z^2\setminus\{0\}$}
\]
being the intersection of two shifted unit disks.

We point out that $g_{j_1,j_2}$ is a continuous and compactly supported 
function, and hence, the sum over $k\in\Z^2$ in \req{techlattice1} contains
only a finite number of nonzero terms for each pair
of grid indices $j_1,j_2\in\Z^2\setminus\{0\}$. 
Moreover, the Fourier transform of 
\[
   g_{j_1,j_2}(k) \,=\, \int_{\R^2} 1_{\Omega_{j_2}}(x+k/\ell)\,1_{\Omega_{j_1}}(x)\dx 
   \,=\, (1_{\Omega_{j_2}}*1_{\Omega_{j_1}})(k/\ell)
\]
is given by
\be{ghatxi}
   \widehat{g}_{j_1,j_2}(\xi) \,=\,
   \ell^2\,\widehat{1}_{\Omega_{j_1}}(\ell\xi)\,\widehat{1}_{\Omega_{j_2}}(\ell\xi)\,, 
   \qquad \xi\in\R^2\,.
\ee
For $j\in\Z^2\setminus\{0\}$ and $\xi\neq 0$ the divergence theorem gives
\be{einshatG}
\begin{aligned}
   \widehat{1}_{\Omega_j}(\xi)
   &\,=\, \int_{\Omega_j} e^{-2\pi\rmi \xi\cdot x}\dx
    \,=\, \frac{\rmi}{2\pi|\xi|^2} 
          \int_{\partial \Omega_j} e^{-2\pi\rmi \xi\cdot x} \,\xi\cdot\nu\ds\\[1ex]
   &\,=\, \frac{\rmi}{2\pi|\xi|^2} 
          \int_{\partial \Omega_j^+} e^{-2\pi\rmi \xi\cdot x} \xi\cdot\nu\ds
          \,+\, \frac{\rmi}{2\pi|\xi|^2} 
          \int_{\partial \Omega_j^-} e^{-2\pi\rmi \xi\cdot x} \,\xi\cdot\nu\ds\,,
\end{aligned}
\ee
where $\nu$ denotes the outer normal of $\Omega_j$, and $\partial \Omega_j^\pm$ are the
two pieces of the boundary of $\Omega_j$, for which $\nu\cdot j$ is positive,
respectively negative.
Denoting the opening angle of the two arcs $\partial \Omega_j^\pm$ by $2\alpha$,
and the angle between $j$ and $\xi$ by $\beta$, we can rewrite the
two boundary integrals as
\be{boundary-integral}
   \int_{\partial \Omega_j^\pm} e^{-2\pi\rmi \xi\cdot x}\,\xi\cdot\nu\ds
   \,=\, |\xi| e^{\pm\pi\rmi (|\xi||j|/\ell)\cos\beta}
         \int_{-\alpha}^\alpha e^{-2\pi\rmi|\xi|\cos(\beta-t)}
         \cos(\beta-t)\dt\,.
\ee
In this form the integrals are amenable to an application of the van der Corput
lemma, cf., e.g., Stein~\cite{Stei93}: 
Decomposing the domain $(-\alpha,\alpha)$ into subintervals where
\\[-0.5ex]
\begin{itemize}
    \item[(i)] \qquad
               ${\displaystyle
                2\,\bigl|\cos(\beta-t)\bigr|
                \,=\, \Bigl|\frac{\rmd^2}{\dt^2}\bigl(2\cos(\beta-t)\bigr)
                \Bigr| \,\geq\, 1}$\,,\\[-0.5ex]
\end{itemize}
or where\\[-0.5ex]
\begin{itemize}
    \item[(ii)] \qquad
                ${\displaystyle
                 2\,\bigl|\sin(\beta-t)\bigr| 
                 \,=\, \Bigl|\frac{\rmd}{\dt}\bigl(2\cos(\beta-t)\bigr)\Bigr|}$
                 \quad \text{is monotone and greater than one}\,,\\[-0.5ex]
\end{itemize}
we obtain an inequality of the form
\[
   \int_{-\alpha}^\alpha e^{-2\pi\rmi|\xi|\cos(\beta-t)}\cos(\beta-t)\dt
   \,\leq\, C\,|\xi|^{-1/2} \qquad \text{for \ $|\xi|\geq 1$}\,,
\]
with some uniform constant $C>0$, which is independent of $\alpha$, $\beta$,
and hence, independent of $\xi,j\in\R^2$ with $|\xi|,|j|\geq 1$. Inserting this
inequality into \req{boundary-integral} and \req{einshatG} we thus arrive at
\[
   \bigl|\widehat{1}_{\Omega_j}(\xi)\bigr|
   \,\leq\, \frac{C}{\pi}\,|\xi|^{-3/2}\,, \qquad |\xi|\geq 1\,,
\]
and it therefore follows from \req{ghatxi} that
\be{ghatbound}
   \bigl|\widehat{g}_{j_1,j_2}(\xi)\bigr|
   \,\leq\, (C/\pi)^2\,\ell^{-1}\,|\xi|^{-3}\,, \qquad |\xi|\geq 1\,.
\ee
Accordingly, we can use the Poisson summation formula 
to rewrite \req{techlattice1} as
\be{techlattice2}
   \E\bigl[V_\Lambda^2\bigr] 
   \,=\, \ell^2 \!\!\!
         \sum_{j_1,j_2\in\Z^2\setminus\{0\}} \testfct(j_1)\testfct(j_2)
         \sum_{k\in\Z^2} \widehat{g}_{j_1,j_2}(k)\,.
\ee

On the other hand we deduce from \req{VarV-lattice} that
\begin{align*}
   \E\bigl[V_\Lambda\bigr]^2
   &\,=\, \sum_{j_1,j_2\in\Z^2\setminus\{0\}} \testfct(j_1)\testfct(j_2)\,
          \bigl|\Lambda\cap(\Lambda-j_1)\bigr|
          \bigl|\Lambda\cap(\Lambda-j_2)\bigr|\\[1ex]
   &\,=\, \ell^4 \!\!\!
          \sum_{j_1,j_2\in\Z^2\setminus\{0\}}\testfct(j_1)\testfct(j_2)
          \,\bigl|\Omega_{j_1}\bigr|\bigl|\Omega_{j_2}\bigr|
    \,=\, \ell^4 \!\!\!
          \sum_{j_1,j_2\in\Z^2\setminus\{0\}}\testfct(j_1)\testfct(j_2)
          \,\widehat{1}_{\Omega_{j_1}}(0)\,\widehat{1}_{\Omega_{j_2}}(0)\\[1ex]
   &\,=\, \ell^2 \!\!\!
          \sum_{j_1,j_2\in\Z^2\setminus\{0\}}\testfct(j_1)\testfct(j_2)
          \,\widehat{g}_{j_1,j_2}(0)\,.
\end{align*}
Together with \req{techlattice2} we therefore conclude that
\[
   \Var\bigl[V_\Lambda\bigr]
   \,=\, \ell^2 \!\!\!
         \sum_{j_1,j_2\in\Z^2\setminus\{0\}}\testfct(j_1)\testfct(j_2)
         \sum_{k\in\Z^2\setminus\{0\}} \widehat{g}_{j_1,j_2}(k)\,,
\]
and hence, \req{ghatbound} yields
\[
   \frac{\Var\bigl[V_\Lambda\bigr]}{|\Lambda|}
   \,\leq\, (C/\pi)^2\,\frac{\ell}{|\Lambda|}\, 
            \Bigl(\sum_{k\in\Z^2\setminus\{0\}} |k|^{-3}\Bigr)
            \Bigl(\sum_{j\in\Z^2\setminus\{0\}} |v(j)|\Bigr)^2
   \,=\, O(1/\ell)\,, \qquad \ell\to\infty\,.
\]
Accordingly we have shown for the two-dimensional stationary lattice $\LL$, 
that $\Var\bigl[V_\Lambda\bigr]/|\Lambda|\to 0$ as $\ell\to\infty$.

\begin{remark}
\label{Rem:Lotz26}
\rm
For the one-dimensional stationary lattice one can show with the same line of
argument -- in this case $\Omega_j$ is an interval -- that $\Var\bigl[V_\Lambda\bigr]$
is bounded independent of $\ell>0$.
In higher dimensions ($d\geq 3$) the lack of smoothness of the indicator 
functions of intersections of balls makes this approach artificially difficult. 
Instead one can use smooth approximations of $1_\Lambda$, as suggested by 
Lotz and Klatt in \cite{Lotz26} in the context of persistence of hyperuniformity, 
to extend this analysis to higher dimensions.
\fin
\end{remark}

\subsection{Determinantal point processes with projection kernels}
\label{subsec:DPP-technical}
Here we rewrite the representation from Proposition~\ref{prop:chi4} of the 
limit $\Var\bigl[V_\Lambda\bigr]/|\Lambda|$ for determinantal point processes.

\begin{lemma}\label{Lem:NHUDPP}
For the determinantal point process $\KK$ defined in \req{rhodpp} 
and for $V_\Lambda$ given by \req{G} for some even function 
$v\in L^1(\R^d)\cap L^2(\R^d)$ there holds
\[
   \lim_{l\to \infty}\frac{1}{|\Lambda|}\Var \left[\Testfct_\Lambda\right]
   \,=\, 2\int_{\R^d}
             \big(\Ghat(\xi)\big)^2\,\bigl|(E+\xi)\setminus E\bigr|^2
          \dxi  
         \,-\, 2 \int_{(\R^d)^2} 
                    \Ghat(\xi_1)\Ghat(\xi_2)\,
    \bigl|\,(E+\xi_1)\cap(E+\xi_2)\setminus(E+\xi_1+\xi_2)\setminus E\,\bigr|
    \rmd(\xi_1,\xi_2)\,.
\]
\end{lemma}

\begin{proof}
Let $v\in L^1(\R^d)\cap L^2(\R^d)$ be an even function. According to \req{rhodpp} we have 
\[
   \rho^{(3)}(0,x_1,x_2)
   \,=\, 1 \,+\, 2K(x_1)K(x_2)K(x_1-x_2)
         \,-\, K^2(x_1) \,-\, K^2(x_2) \,-\, K^2(x_1-x_2)\,,
\]
and hence,
\be{dpplimes-teil1}
\begin{aligned}
   &2\int_{\R^d} v^2(x)\rho_2(x)\dx
   \,+\, 4\int_{(\R^d)^2}v(x_1)v(x_2)\rho^{(3)}(0,x_1,x_2)\dxx_2 \\[1ex]
   &\qquad\,=\,
   2\,\norm{v}^2_{L^2} \,-\, 2\,\norm{Kv}_{L^2}^2
   \,+\, 4\left(\int_{\R^d}v(x)\dx\right)^2\\[1ex]
   &\qquad\phantom{\,=}\
   \,+\, 8 \int_{\R^d} K(x)v(x)\bigl(K*(Kv)\bigr)(x)\dx
   \,-\, 8\int_{\R^d} v(x)\dx\int_{\R^d}K^2(x)v(x)\dx
   \,-\, 4 \int_{\R^d} v(x)(K^2*v)(x)\dx\,.
\end{aligned}
\ee

On the other hand, it follows from \req{chi4} and \req{rhodpp} that
\begin{align*}
    \chi^{(4)}(\xx_4)
    &\,=\, -K(x_2-x_3)^2 \,-\, K(x_2-x_4)^2 \,-\, K(x_1-x_3)^2-K(x_1-x_4)^3 
           \\[1ex]
    &\phantom{\,=\ }
           \,+\, 2 K(x_1-x_2)K(x_2-x_3)K(x_1-x_3)
           \,+\, 2 K(x_1-x_2)K(x_2-x_4)K(x_1-x_4)\\[1ex]
    &\phantom{\,=\ }
           \,+\,2 K(x_1-x_3)K(x_3-x_4)K(x_1-x_4) 
           \,+\,2 K(x_2-x_3)K(x_2-x_4)K(x_3-x_4)\\[1ex]
    &\phantom{\,=\ }
           \,+\, K(x_2-x_4)^2K(x_1-x_3)^2
           \,+\, K(x_2-x_3)^2K(x_1-x_4)^2 \\[1ex]
    &\phantom{\,=\ }
           \,-\, 2 K(x_1-x_2)K(x_2-x_4)K(x_1-x_3)K(x_3-x_4)
           \,-\, 2K(x_1-x_2)K(x_2-x_3)K(x_3-x_4)K(x_1-x_4)\\[1ex]
    &\phantom{\,=\ }
           \,-\, 2K(x_2-x_3)K(x_2-x_4)K(x_1-x_3)K(x_1-x_4)\,.
\end{align*}
Accordingly, exploiting the symmetry of $K$ and $v$, there holds
\begin{align*}
   &\int_{(\R^d)^3}v(x_1)v(x_2)\chi^{(4)}(x_1,0,x_3,x_2+x_3)\dxx_3
    \\[1ex]
   &\qquad\,=\, 
    -\int_{(\R^d)^3} v(x_1)v(x_2)
              \Bigl(K^2(x_3) \,+\, K^2(x_2+x_3) \,+\, K^2(x_1-x_3)
                    \,+\, K^2(x_1-x_2-x_3)\Bigr)\dxx_3\\[1ex]
   &\qquad\phantom{\,=\ }
       \,+\, 2\int_{(\R^d)^3} v(x_1)v(x_2)
                \Bigl(K(x_1)K(x_3)K(x_1-x_3) 
                      \,+\, K(x_1)K(x_2+x_3)K(x_1-x_2-x_3)\\
   &\qquad\phantom{\,=\ 
       \,+\, 2\int_{(\R^d)^3} v(x_1)v(x_2)\Bigl(}
             +\, K(x_3-x_1)K(x_2)K(x_1-x_2-x_3) 
             \,+\, K(x_2)K(x_3)K(x_2+x_3)\Bigr)\\[1ex]
   &\qquad\phantom{\,=\ }
       \,+\, \int_{(\R^d)^3} v(x_1)v(x_2)
                \Bigl(K^2(x_2+x_3)K^2(x_1-x_3) 
                      \,+\, K^2(x_3)K^2(x_1-x_2-x_3)\Bigr)\dxx_3\\[1ex]
   &\qquad\phantom{\,=\ }
       \,-\, 2\int_{(\R^d)^3} v(x_1)v(x_2)
                \Bigl(K(x_1)K(x_2)K(x_1-x_3)K(x_2+x_3) 
                   \,+\, K(x_1)K(x_2)K(x_3)K(x_1-x_2-x_3)\Bigr)\dxx_3
                   \\[1ex]
   &\qquad\phantom{\,=\ }
       \,-\, 2\int_{(\R^d)^3} v(x_1)v(x_2)
                K(x_3)K(x_2+x_3)K(x_1-x_3)K(x_1-x_2-x_3)\dxx_3\,.
\end{align*}
It turns out that the integrals in the individual lines all have
the same values, respectively, giving
\begin{align*}
   &\int_{(\R^d)^3}v(x_1)v(x_2)\chi^{(4)}(x_1,0,x_3,x_2+x_3)\dxx_3\\[1ex]
   &\qquad\,=\, -4\,\Bigl(\int_{\R^d} v(x)\dx\Bigr)^2 \norm{K}_{L^2}^2
          \,+\, 8 \int_{\R^d}v(x)\dx \int_{\R^d}v(x)K(x)(K*K)(x)\dx
          \\[1ex]
   &\qquad\phantom{\,=\ }
          \,+\, 2 \int_{\R^d} v(x)(K^2*K^2*v)(x)\dx
          \,-\, 4\int_{\R^d} v(x)K(x)
              \bigl(K*K*(Kv)\bigr)(x)\dx\\[1ex]
   &\qquad\phantom{\,=\ }
       \,-\, 2\int_{(\R^d)^3} v(x_1)v(x_2)
                K(x_3)K(x_2+x_3)K(x_1-x_3)K(x_1-x_2-x_3)\dxx_3\,.
\end{align*}
Inserting this result into \req{Hank17b}, together with \req{dpplimes-teil1},
and making use of the fact that $\norm{K}_{L^2}=1$ and $K*K=K$, 
cf.~Section~\ref{ss:DPP}, we obtain
\begin{subequations}
\label{eq:limit-tmp}
\begin{align}
\label{eq:limit-tmp1}
   \lim_{l\to \infty}\frac{1}{|\Lambda|}\Var \left[\Testfct_\Lambda\right]
   &\,=\, 2\,\norm{v}_{L^2}^2 \,-\, 4 \int_{\R^d} v(x)(K^2*v)(x)\dx
          \,+\, 2 \int_{\R^d} v(x)(K^2*K^2*v)(x)\dx\\[1ex]
\label{eq:limit-tmp2}
   &\phantom{\,=\ }
    -\, 2\,\norm{Kv}_{L^2}^2
           \,+\, 4 \int_{\R^d} K(x)v(x)\bigl(K*(Kv)\bigr)(x)\dx   \\[1ex]
\label{eq:limit-tmp3}
   &\phantom{\,=\ }
    -\, 2\int_{(\R^d)^3} v(x_1)v(x_2)
                K(x_3)K(x_2+x_3)K(x_1-x_3)K(x_1-x_2-x_3)\dxx_3\,.
\end{align}
\end{subequations}
The three lines of \req{limit-tmp} will now be treated separately; denote
their values by $l_1$, $l_2$, and $l_3$, respectively.
Concerning~\req{limit-tmp1} we use the Plancherel identity and the
convolution theorem to obtain
\[
   l_1 \,=\, 2\int_{\R^d}
                \bigl(\Ghat(\xi)\bigr)^2\bigl(1\,-\, \widehat{K^2}(\xi)\bigr)^2 
              \dxi\,,
\]
where  
\[
   \widehat{K^2}(\xi) \,=\, \int_{\R^d} \Khat(\xi')\Khat(\xi-\xi')\dxi'
   \,=\, \bigl| E \cap (E+\xi) \bigr|\,,
\]
because $\Khat=1_E$ and $E$ is symmetric with respect to the origin. 
Moreover, since $|E|=1$ this gives
\begin{subequations}
\label{eq:limit-erg}
\be{limit-erg1}
   l_1 \,=\, 2 \int_{\R^d} \bigl(\Ghat(\xi)\bigr)^2\,
                   \bigl|(E+\xi)\setminus E\bigr|^2\dxi\,.
\ee

As far as \req{limit-tmp2} is concerned, the Plancherel identity gives
\begin{align*}
   l_2 \,=\, \int_{\R^d} (Kv)(x)\bigl(4 K*(Kv)(x) - 2(Kv)(x)\bigr)\dx
   \,=\, \int_{\R^d} \widehat{Kv}(\xi)\bigl(4\Khat(\xi)-2\bigr)\widehat{Kv}(\xi)\dxi\,.
\end{align*}
Rewriting the Fourier transform of $Kv$ in terms of the convolution integral
\[
   \widehat{Kv}(\xi) \,=\, \int_{\R^d} \Ghat(\xi')\Khat(\xi-\xi')\dxi'\,,
\]
and using the short-hand notation $\xifett_n$
for $(\xi_1,\dots,\xi_n)\in(\R^d)^n$, $n\in\N$, it follows that
\begin{align}
\nonumber
   l_2 &\,=\, \int_{(\R^d)^3} 
                 \Ghat(\xi_1)\Ghat(\xi_2) \Khat(\xi_3-\xi_1)
                 \bigl(4\Khat(\xi_3)-2\bigr)\Khat(\xi_3-\xi_2)\dxifett_3\\[1ex]
\nonumber
   &\,=\, 2\int_{(\R^d)^3} 
             \Ghat(\xi_1)\Ghat(\xi_2) 
             \Khat(\xi_3-\xi_1)\Khat(\xi_3)\Khat(\xi_3-\xi_2)\dxifett_3\\[1ex]
\nonumber
   &\phantom{\,=\ }
          -\, 2\int_{(\R^d)^3}
             \Ghat(\xi_1)\Ghat(\xi_2) 
             \Khat(\xi_3-\xi_1)\bigl(1-\Khat(\xi_3)\bigr)
             \Khat(\xi_3-\xi_2)\dxifett_3\\[1ex]
\nonumber
   &\,=\, 2\int_{(\R^d)^3} 
             \Ghat(\xi_1)\Ghat(\xi_2) 
             \Khat(\xi_3-\xi_1)\Khat(\xi_3)\Khat(\xi_3-\xi_2)\dxifett_3\\[1ex]
\nonumber
   &\phantom{\,=\ }
          -\, 2\int_{(\R^d)^3}
             \Ghat(\xi_1)\Ghat(\xi_2) 
             \Khat(\xi_3+\xi_2)\bigl(1-\Khat(\xi_3+\xi_1+\xi_2)\bigr)
             \Khat(\xi_3+\xi_1)\dxifett_3\\[1ex]
\label{eq:limit-erg2}
   &\,=\, 2\int_{(\R^d)^2}
              \Ghat(\xi_1)\Ghat(\xi_2)
              \Bigl(\bigl|\,(E+\xi_1)\cap E\cap(E+\xi_2)\,\bigr|
                 \,-\, \bigl|\,(E+\xi_1)\cap(E+\xi_2)\setminus (E+\xi_1+\xi_2)\,
                 \bigr| 
              \Bigr)\dxifett_2\,,
\end{align}
where we have used in the final step that $\Ghat$ is an even function.

Finally, we consider the integrand of \req{limit-tmp3}
as a product of the first three functions and the remaining three, 
compute their $3d$-dimensional Fourier transforms and use the corresponding
Plancherel identity to rewrite 
\be{limit-erg3}
\begin{aligned}
   l_3 &\,=\, -2 \int_{(\R^d)^2} 
                    \Ghat(\xi_1)\Ghat(\xi_2)
                 \int_{\R^d}\Khat(\xi_3)
                    \Khat(\xi_1+\xi_3)\Khat(\xi_3-\xi_2)\Khat(\xi_1+\xi_3-\xi_2)
                 \dxi_3\dxifett_2\\[1ex]
       &\,=\, -2 \int_{(\R^d)^2}
                    \Ghat(\xi_1)\Ghat(\xi_2)\,
                    \bigl|\,E\cap(E+\xi_1)\cap(E+\xi_2)\cap(E+\xi_1+\xi_2)\,
                    \bigr|
                 \dxifett_2\,,
\end{aligned}
\ee
where we have used once again that $\Ghat(\xi_1)=\Ghat(-\xi_1)$.
\end{subequations}

Inserting \req{limit-erg} into \req{limit-tmp} we finally arrive at
\begin{align*}
   &\lim_{l\to \infty}\frac{1}{|\Lambda|}\Var \left[\Testfct_\Lambda\right]
    \,=\, l_1 + l_2 + l_3\\[1ex]
   &\qquad 
    \,=\, 2 \int_{\R^d} 
               \bigl(\Ghat(\xi)\bigr)^2\,\bigl|(E+\xi)\setminus E\bigr|^2
            \dxi
          \,-\, 2 \int_{(\R^d)^2} 
                     \Ghat(\xi_1)\Ghat(\xi_2)\,
                     \bigl|\,(E+\xi_1)\cap(E+\xi_2)
                             \setminus E\setminus(E+\xi_1+\xi_2)\,\bigr|
                  \dxifett_2\,,
\end{align*}
which was to be shown.
\end{proof}

\color{black}

\subsection{Gibbs point processes with superstable pair interactions}
\label{subsec:Gibbs-technical}
Here we start with a bound for the expected interaction energy
of a given additional configuration.

\begin{lemma}
\label{Lem2.1}
Assume that $\GG$ is a $(\beta,z,u)$-Gibbs point process with a superstable pair potential $u$. 
Denote by $q>0$ the parameter in the Ruelle bound~\req{R} and by $B$ the constant in the
stability bound~\req{B}. Then there holds
\begin{align}\label{eq:eminusWbd}
    \E \bigl[e^{ -\beta W(\xx_n\mid \gamma  ) } \bigr]
    \,\leq\, \exp\bigl(nq\,e^{2n\beta B} \norm{f}_{L^1} \bigr)
\end{align}
for every $n\in\N$ and $\xx_n\in(\R^d)^n$.
\end{lemma}

\begin{proof}
Using the Laplace functional of $\GG$ we can rewrite
\begin{align*}
    \E\left[e^{ -\beta W(\xx_{n}\mid \gamma)}\right] 
    \,=\, 
    \sum_{k=0}^\infty \frac{1}{k!}\int_{(\R^d)^k}
    \prod_{j=1}^k
    \left(e^{ -\beta W(\xx_n\mid y_j)}-1
    \right) \rho^{(k)}(\yy_k) \dyy_k\,,
\end{align*}  
cf, e.g., Jansen~\cite[Theorem~2.40]{jansen18}.
Then the assertion~\req{eminusWbd} follows from the well-known inequality
\[
    \int_{\R^d}\Bigl|e^{-\beta W(\xx_n\mid y)}-1\Bigr|\dy
    \,\leq\, n e^{2n\beta B} \norm{f}_{L^1}
\]
and the Ruelle bound~\req{R} for the correlation functions.
\end{proof}

In the following results we will make use of the abbreviations
\[
   u_{ij} = u(x_i-x_j) \qquad \text{and} \qquad f_{ij} = f(x_i-x_j)
\]
for $i,j\in\N$.

\begin{lemma}\label{Lem:mixedexpectations1}
For $V_\Lambda$, $V_\Lambda'$, and $V_\Lambda''$ of \req{G},
\req{2ndorderintegral2}, and \req{2ndorderintegral3}, respectively, 
there hold
\begin{align*}
    \E \bigl[V_\Lambda V_\Lambda'\bigr] 
    &\,=\, \int_{\Lambda^4}  
               v(x_1-x_2)v(x_3-x_4)\rho^{(4)}(\xx_4)\dxx_4
           \,+\, 2 \int_{\Lambda^3} 
                      v(x_1-x_2)v(x_1-x_3)\rho^{(3)}(\xx_{3})\dxx_{3}\,,\\[1ex]
    \E \bigl[V_\Lambda V_\Lambda''\bigr] 
    &\,=\, \int_{\Lambda^4} v(x_1-x_2)v(x_3-x_4)\rho^{(4)}(\xx_4)\dxx_{4}\,,
           \\[1ex]
   \E\bigl[(V_\Lambda')^2\bigr]
   &\,=\, \int_{\Lambda^4} v(x_3-x_1)v(x_3-x_2) \rho^{(4)}(\xx_4)\dxx_4
          \,+\, \int_{\Lambda^3} v(x_3-x_1)v(x_3-x_2)\rho^{(3)}(\xx_3)\dxx_3
          \,-\, \mathcal{R}_1 \,-\, \mathcal{R}_2\,,\\[1ex]
   \E\bigl[V_\Lambda'V_\Lambda''\bigr]
   &\,=\, \int_{\Lambda^4} v(x_4-x_3)v(x_1-x_2) \rho^{(4)}(\xx_4)\dxx_4
          \,-\, \mathcal{R}_3\,,\\[1ex]
   \E \bigl[(V_\Lambda'')^2\bigr]
   &\,=\, \int_{\Lambda^4} v(x_1-x_2)v(x_3-x_4) \rho^{(4)}(\xx_4)\dxx_4
          \,-\, \mathcal{R}_4\,,
\end{align*}
where
\begin{subequations}
\label{eq:R1bis4}
\begin{align}\label{eq:R0}
    \mathcal{R}_1
    &\,=\, z^3 \int_{\Lambda^3} 
                  v(x_3-x_1)v(x_3-x_2)\,f_{12}\,
                  e^{-\beta u_{13}-\beta u_{23}}\,
                  \E\bigl[e^{-\beta W(\xx_3\mid\gamma)}\bigr]\dxx_3 \\[1ex]
    \label{eq:R1}
    \mathcal{R}_2
    &\,=\, z^4 \int_{\Lambda^4} 
                  v(x_3-x_1)v(x_4-x_2)\,f_{12}\,
                  e^{-\beta u_{13}-\beta u_{14}-\beta u_{23} 
                     -\beta u_{24}-\beta u_{34}}\,
                  \E\bigl[e^{-\beta W(\xx_4\mid\gamma)}\bigr]\dxx_4 \,,\\[1ex]
    \label{eq:R2}
    \mathcal{R}_3
    &\,=\,
    z^4\int_{\Lambda^4}
          v(x_4-x_3)v(x_1-x_2)\, 
          e^{-\beta u_{12}-\beta u_{14}-\beta u_{24}-\beta u_{34}}\,
          \bigl(e^{-\beta W(\xx_2\mid x_3)}-1\bigr)\,
          \E \bigl[e^{-\beta W(\xx_{4}\mid\gamma)} \bigr]\dxx_{4} \,,\\[1ex]
    \label{eq:R3}
    \mathcal{R}_4
    &\,=\,
    z^4\int_{\Lambda^4}
          v(x_3-x_4)v(x_1-x_2)\, 
          e^{-\beta u_{12}-\beta u_{34}}\,
          \bigl(e^{-\beta W(\xx_2\mid x_3,x_4)}-1\bigr)\,
          \E \bigl[e^{-\beta W(\xx_{4}\mid\gamma)} \bigr]\dxx_{4} \,.
\end{align}
\end{subequations}
\end{lemma}

\begin{proof}
According to \req{MGNZ} we have
\begin{align*}
    \E\bigl[V_\Lambda V_\Lambda'\bigr]
    &\,=\,\E\left[ 
            \sum_{x_1 \neq x_2 \in \gamma }\chiL(\xx_2) v(x_1-x_2)
            \sum_{x_4\in\gamma }\chiL(x_4)\,z 
            \int_\Lambda v(x_3-x_4) e^{-\beta W(x_3 \mid \gamma)} \dx_3
            \right] \\[1ex]
    &\,=\,\E\left[
            z^3\int_{\Lambda^2} v(x_1-x_2)
                  \sum_{x_4\in\gamma\cup\{x_1,x_2\}}\chiL(x_4) 
                  \int_\Lambda 
                     v(x_3-x_4) e^{-\beta W(x_3 \mid \gamma\cup\{x_1,x_2\})} 
                  \dx_3\
                  e^{-\beta H(\xx_2)-\beta W(\xx_2 \mid \gamma)}  \dxx_2
            \right] \\[1ex]
    &\,=\,\E\left[
            \sum_{x_4\in\gamma}\chiL(x_4) \,z^3 
            \int_{\Lambda^3}
               v(x_1-x_2)v(x_3-x_4) 
               e^{-\beta H(\xx_3)-\beta W(\xx_3\mid\gamma)}\dxx_3
            \right] \\[1ex]
    &\phantom{\,=} \ +\, 
          2\,\E\left[
               z^3 
               \int_{\Lambda^3}
                  v(x_1-x_2)v(x_1-x_3) 
                  e^{-\beta H(\xx_3)-\beta W(\xx_3\mid\gamma)}\dxx_3
               \right].
\end{align*}
With \req{GNZ} and \req{rho-n} we therefore get
\begin{align*}
    \E\bigl[V_\Lambda V_\Lambda'\bigr]
    &\,=\,\E\left[
            z^4\int_{\Lambda^4} 
                  v(x_1-x_2)v(x_3-x_4)\,
                  e^{-\beta H(\xx_3)-\beta W(\xx_3\mid\gamma\cup\{x_4\})}
                  e^{-\beta W(x_4\mid\gamma)}\dxx_4
            \right] \\[1ex]
    &\phantom{\,=}
            \ +\, 2\int_{\Lambda^3} v(x_1-x_2)v(x_1-x_3)\rho^{(3)}(\xx_3)\dxx_3
          \\[1ex]
    &\,=\,\int_{\Lambda^4}  
             v(x_1-x_2)v(x_3-x_4)\rho^{(4)}(\xx_4)\dxx_4
          \,+\, 2 \int_{\Lambda^3} 
                      v(x_1-x_2)v(x_1-x_3)\rho^{(3)}(\xx_{3})\dxx_{3}\,.
\end{align*}
Likewise we obain from \req{MGNZ} and \req{rho-n} that
\begin{align*}
    \E \bigl[V_\Lambda V_\Lambda''\bigr] 
    &\,=\,\E\Biggl[
            \sum_{x_3 \neq x_4 \in \gamma}\chiL(x_3)\chiL(x_4) v(x_3-x_4)
            \int_{\Lambda^2}
               v(x_1-x_2) z^2 e^{-\beta H(\xx_2)-\beta W(\xx_2\mid\gamma)} 
            \dxx_2
            \Biggr] \\[1ex]
    &\,=\,\E\left[
            z^4 \int_{\Lambda^4} 
                   v(x_3-x_4)v(x_1-x_2)
                   e^{-\beta H(\xx_2)-\beta W(\xx_2\mid\gamma\cup\{x_3,x_4\})}
                   e^{-\beta u(x_3-x_4)-\beta W(x_3,x_4\mid\gamma)}
                \dxx_{4}
            \right]\\[1ex]
    &\,=\, \int_{\Lambda^4} v(x_1-x_2)v(x_3-x_4)\rho^{(4)}(\xx_4)\dxx_{4}\,. 
\end{align*}
Next we compute
\begin{align*}
    \E\bigl[(V_\Lambda')^2\bigr]
    &\,=\,\E\Biggl[
            \sum_{x_3,x_4\in\gamma}\chiL(x_3)\chiL(x_4)
            \,z^2 \int_{\Lambda^2} 
                     v(x_3-x_1)v(x_4-x_2)
                     e^{-\beta W(\xx_2\mid\gamma)} \dxx_2
            \Biggr] \\[1ex]
    &\,=\,\E\Biggl[
            \sum_{x_3\neq x_4\in\gamma}\chiL(x_3)\chiL(x_4)
            \,z^2 \int_{\Lambda^2} 
                     v(x_3-x_1)v(x_4-x_2)
                     e^{-\beta W(\xx_2\mid\gamma)} \dxx_2
            \Biggr] \\[1ex]
    &\phantom{\,=} \
           +\, \E\left[
                 \sum_{x_3\in\gamma}\chiL(x_3)
                 \,z^2 \int_{\Lambda^2} 
                     v(x_3-x_1)v(x_3-x_2)
                     e^{-\beta W(\xx_2\mid\gamma)} \dxx_2
                 \right] ,
\end{align*}
and with \req{MGNZ} and \req{GNZ} we conclude that
\begin{align*}
    \E\bigl[(V_\Lambda')^2\bigr]
    &\,=\,\E\left[
             z^4 \int_{\Lambda^4}
                    v(x_3-x_1)v(x_4-x_2)
                    e^{-\beta W(\xx_2\mid\gamma\cup\{x_3,x_4\})}
                    e^{-\beta H(x_3,x_4)-\beta W(x_3,x_4\mid\gamma)}\dxx_4
             \right] \\[1ex]
     &\phantom{\,=} \
            +\, \E\left[
                  z^3 \int_{\Lambda^3} 
                         v(x_3-x_1)v(x_3-x_2) 
                         e^{-\beta W(\xx_2\mid\gamma\cup\{x_3\})}
                         e^{-\beta W(x_3\mid\gamma)}\dxx_3
                  \right] \\[1ex]
    &\,=\,\E\left[
             z^4 \int_{\Lambda^4}
                    v(x_3-x_1)v(x_4-x_2)
                    e^{-\beta H(\xx_4) - \beta W(\xx_4\mid\gamma)}
                    e^{\beta u(x_1-x_2)}\dxx_4
             \right] \\[1ex]
     &\phantom{\,=} \
            +\, \E\left[
                  z^3 \int_{\Lambda^3} 
                         v(x_3-x_1)v(x_3-x_2) 
                         e^{-\beta H(\xx_3)-\beta W(\xx_3\mid\gamma)}
                         e^{\beta u(x_1-x_2)}\dxx_3
                  \right] .
\end{align*}
Using the Mayer function~\req{f} we have 
$e^{\beta u(x_1-x_2)}=1-e^{\beta u(x_1-x_2)}f(x_1-x_2)$,
and \req{rho-n} therefore gives
\[
   \E\bigl[(V_\Lambda')^2\bigr]
   \,=\, \int_{\Lambda^4} v(x_3-x_1)v(x_4-x_2) \rho^{(4)}(\xx_4)\dxx_4
         \,+\, \int_{\Lambda^3} v(x_3-x_1)v(x_3-x_2)\rho^{(3)}(\xx_3)\dxx_3
         \,-\, \mathcal{R}_1 \,-\, \mathcal{R}_2
\]
with $\mathcal{R}_1$ and $\mathcal{R}_2$ as specified in 
\req{R1bis4}. In a similar fashion we deduce with the help of \req{GNZ} that
\begin{align*}
    \E \bigl[V_\Lambda' V_\Lambda''\bigr] 
    &\,=\,\E\left[
            \sum_{x_4 \in \gamma}\chiL(x_4)
               z^3 \int_{\Lambda^3} v(x_4-x_3)v(x_1-x_2) 
                      e^{-\beta W(x_3\mid\gamma)}
                      e^{-\beta H(\xx_2)-\beta W(\xx_2\mid\gamma)}\dxx_3
            \right] \\[1ex]
    &\,=\,\E\left[
            z^4 \int_{\Lambda^4} 
                   v(x_4-x_3)v(x_1-x_2)
                   e^{-\beta W(x_3\mid\gamma\cup\{x_4\})} 
                   e^{-\beta H(\xx_2)-\beta W(\xx_2\mid\gamma\cup\{x_4\})}
                   e^{-\beta W(x_4\mid\gamma)}\dxx_4
            \right] \\[1ex]
    &\,=\,\E\left[
            z^4 \int_{\Lambda^4} 
                   v(x_4-x_3)v(x_1-x_2)
                   e^{-\beta H(\xx_4)-\beta W(\xx_4\mid\gamma)}
                   e^{\beta u(x_1-x_3)+\beta u(x_2-x_3)}\dxx_4
            \right].
\end{align*}
Therefore, writing 
\[
   e^{\beta u(x_1-x_3)+\beta u(x_2-x_3)}
   \,=\, 1 \,-\, e^{\beta u(x_1-x_3)}e^{\beta u(x_2-x_3)}
                 \bigl(e^{-\beta W(\xx_2\mid x_3)}-1\bigr)
\]
and using \req{rho-n} we conclude that
\[
    \E \bigl[V_\Lambda' V_\Lambda''\bigr] 
    \,=\, \int_{\Lambda^4} v(x_4-x_3)v(x_1-x_2) \rho^{(4)}(\xx_4)\dxx_4
          \,-\, \mathcal{R}_3
\]
with $\mathcal{R}_3$ of \req{R2}.

Lastly, we calculate
\begin{align*}
    \E \bigl[(V_\Lambda'')^2\bigr]
    &\,=\,\E\left[
            z^4\int_{\Lambda^4}
                  v(x_1-x_2)v(x_3-x_4) 
                  e^{-\beta H(\xx_2)-\beta W(\xx_2\mid\gamma)} 
                  e^{-\beta H(x_3,x_4)-\beta W(x_3,x_4\mid\gamma)} 
                  \dxx_4
            \right]\\[1ex]
    &\,=\,\E\left[
            z^4\int_{\Lambda^4}
                  v(x_1-x_2)v(x_3-x_4) 
                  e^{-\beta H(\xx_4)-\beta W(\xx_4\mid\gamma)} 
                  e^{\beta W(\xx_2\mid x_3,x_4)}
                  \dxx_4
            \right],
\end{align*}
and writing
\[
   e^{\beta W(\xx_2\mid x_3,x_4)} 
   \,=\, 1 - 
         e^{\beta u(x_1-x_3)+\beta u(x_1-x_4)+\beta u(x_2-x_3)+\beta u(x_2-x_4)}
         \bigl(e^{-\beta W(\xx_2\mid x_3,x_4)}-1\bigr)
\]
we finally arrive at
\[
    \E \bigl[(V_\Lambda'')^2\bigr]
    \,=\, \int_{\Lambda^4} v(x_1-x_2)v(x_3-x_4) \rho^{(4)}(\xx_4)\dxx_4
          \,-\, \mathcal{R}_4
\]
with $\mathcal{R}_4$ of \req{R3}.
\end{proof}

\begin{proofof}{Lemma~\ref{Lem:Vtilde}}
Let $V_\Lambda^\eps$ be defined as in \req{Veps}.
Using the symmetry of $v$ and $u$, and the symmetry and translation invariance
of the correlation functions, it follows from 
Lemma~\ref{Lem:mixedexpectations1} and \req{variance2ndord} that
\begin{align*}
   \E\bigl[(V_\Lambda^\eps)^2\bigr]
   &\,=\,  \int_{\Lambda^4}v(x_1-x_2)v(x_3-x_4)\rho^{(4)}(\xx_4)\dxx_4
    \\[1ex]
   &\phantom{\,=} \
           \,+\, \bigl(4(1-\eps)^2 + 8\eps(1-\eps) + 4\eps^2\bigr)
                 \int_{\Lambda^3} v(x_1-x_2)v(x_1-x_3)\rho^{(3)}(\xx_3)\dxx_3
                 \\[1ex]
   &\phantom{\,=} \
           \,+\, 2(1-\eps)^2
                 \int_{\Lambda^2} v(x_1-x_2)^2\rho_2(x_1-x_2)\dxx_2
           \,-\, 4\eps^2 \mathcal{R}_1 \,-\, 4\eps^2 \mathcal{R}_2
           \,+\, 4\eps^2 \mathcal{R}_3 \,-\, \eps^2 \mathcal{R}_4\\[1ex]
   &\,=\, \E\bigl[V_\Lambda^2\bigr] 
          \,-\, 4\eps\int_{\Lambda^2} v(x_1-x_2)^2\rho_2(x_1-x_2)\dxx_2
          \,+\, \eps^2 \mathcal{R}_v
\end{align*}
with
\be{Rv}
   \mathcal{R}_v
   \,=\, 2\int_{\Lambda^2} v(x_1-x_2)^2\rho_2(x_1-x_2)\dxx_2
         \,-\, 4 \mathcal{R}_1 \,-\, 4 \mathcal{R}_2
         \,+\, 4 \mathcal{R}_3 \,-\, \mathcal{R}_4\,.
\ee
This yields \req{VarVeps}, and we are left with estimating
the individual terms in \req{Rv} to verify \req{Rvbound}.

First, it is an obvious consequence of the Cauchy-Schwarz inequality that 
\be{I2bound}
   \int_{\Lambda^2} v(x_1-x_2)^2\rho_2(x_1-x_2)\dxx_2
   \,\leq\, \norm{\rho_2}_{L^\infty} \norm{v}_{L^2}^2 \,|\Lambda|\,.
\ee
Second, as far as $\mathcal{R}_1$ is concerned, it follows from \req{R0}, 
Lemma~\ref{Lem2.1}, and the stability bound~\req{B} that 
\[
  |\mathcal{R}_1| 
  \,\leq\, C_0'\int_{\Lambda^3}
                  \bigl|v(x_3-x_1)\bigr|\bigl|v(x_3-x_2)\bigr|
                  \bigl|f(x_1-x_2)\bigr|\dxx_3
\]
for some suitable constant $C_0'$.
Moreover, since $f\in L^1(\R^d)$ this upper bound can further be estimated by
\be{R0bound}
\begin{aligned}
  |\mathcal{R}_1| 
  &\,\leq\, C_0'\int_{\Lambda^2}
                   \bigl|f(x_1-x_2)\bigr| 
                   \int_\Lambda 
                      \bigl|v(x_3-x_1)\bigr|\bigl|v(x_3-x_2)\bigr|\dx_3
                \dxx_2 \\[1ex]
  &\,\leq\, C_0'\,\norm{v}_{L^2}^2 \int_{\Lambda^2}\bigl|f(x_1-x_2)\bigr|\dxx_2
   \,\leq\, C_0'\,\norm{v}_{L^2}^2 \norm{f}_{L^1}\,|\Lambda|\,.
\end{aligned}
\ee
Finally, after renumbering the integration variables appropriately,
and using the symmetry of $u$ and $v$, we conclude from \req{R1bis4} that
\be{varphi-integrals}
   4\mathcal{R}_3 - 4\mathcal{R}_2 - \mathcal{R}_4
   \,=\, z^4\int_{\Lambda^4}
               v(x_4-x_3)v(x_2-x_1)\,e^{-\beta u_{34}-\beta u_{12}}
               \varphi(\xx_4)\,
               \E\bigl[e^{-\beta W(\xx_4\mid\gamma)}\bigr]\dxx_4
\ee
with
\[
  \varphi(\xx_4) 
  \,=\, 4e^{-\beta u_{14}-\beta u_{24}}
        \bigl(e^{-\beta u_{13} - \beta u_{23}}-1\bigr)
        \,-\, 4f_{24}e^{-\beta u_{14}-\beta u_{23}-\beta u_{13}}
        \,+\, 1 \,-\, 
              e^{-\beta u_{13} - \beta u_{14} - \beta u_{23} - \beta u_{24}}\,.
\]
In terms of the Mayer $f$-function $\varphi$ can be rewritten as
\begin{align*}
   \varphi(\xx_4) 
   &\,=\, 4(1+f_{14})(1+f_{24})(f_{13}+f_{23}+f_{13}f_{23})
          \,-\, 4f_{24}(1+f_{14})(1+f_{23})(1+f_{13})\\[1ex]
   &\phantom{\,=}
          \ +\, 1 
          \,-\, (1+f_{13})(1+f_{14})(1+f_{23})(1+f_{24})
          \\[1ex]
   &\,=\, 3f_{13}+3f_{23}-5f_{24}-f_{14}
          + 3f_{13}f_{23} + 3f_{13}f_{14} + 3f_{14}f_{23} - 5f_{14}f_{24}
          - f_{13}f_{24} - f_{23}f_{24}\\[1ex]
   &\phantom{\,=}
          \ + 3f_{14}f_{13}f_{23} - f_{13}f_{14}f_{24}
          - f_{13}f_{23}f_{24} - f_{14}f_{23}f_{24} 
          - f_{13}f_{14}f_{23}f_{24}\,.
\end{align*}
Looking at the individual terms of $\varphi$ several integrals 
of \req{varphi-integrals} can be seen to cancel each other 
by interchaning the variables $x_1$ and $x_2$, or $x_3$ and $x_4$,
or the variable pairs $(x_1,x_2)$ and $(x_3,x_4)$, respectively.
Eventually, this yields the simplified expression
\[
   4\mathcal{R}_3 - 4\mathcal{R}_2 - \mathcal{R}_4
   \,=\, z^4\int_{\Lambda^4}
               v(x_4-x_3)v(x_2-x_1)\, e^{-\beta u_{34}-\beta u_{12}}
               f_{14}f_{23}(2 - f_{13}f_{24})\,
               \E\bigl[e^{-\beta W(\xx_4\mid\gamma)}\bigr]\dxx_4\,,
\]
and therefore we conclude from Lemma~\ref{Lem2.1} and 
the stability bound~\req{B} of the pair potential that
\[
   \bigl|4\mathcal{R}_3 - 4\mathcal{R}_2 - \mathcal{R}_4\bigr|
   \,\leq\, C_1' \int_{\Lambda^4}
                    \bigl|v(x_4-x_3)\bigr|\bigl|v(x_2-x_1)\bigr|
                    \bigl|f(x_1-x_4)\bigr|\bigl|f(x_2-x_3)\bigr| \dxx_4
\]
for some suitable constant $C_1'$, which is independent of $v$. 
The right-hand side can be further estimated in terms of the convolution
$|f|*|v|$ of $|f|$ and $|v|$, namely
\begin{align*}
   \bigl|4\mathcal{R}_3 - 4\mathcal{R}_2 - \mathcal{R}_4\bigr|
   &\,\leq\, C_1'\int_\Lambda
                 \int_{\R^d} \left(
                 \int_{\R^d} |f(x_2-x_3)||v(x_3-x_4)|\dx_3
                 \int_{\R^d} |v(x_2-x_1)||f(x_1-x_4)|\dx_1 \right)
                 \dx_2\dx_4 \\[1ex]
   &\,=\, C_1'\int_\Lambda \int_{\R^d} 
                 \bigl(|f|*|v|\bigr)(x_2-x_4)\,\bigl(|f|*|v|\bigr)(x_2-x_4)
              \dx_2\dx_4\\[1ex]
   &\,=\, C_1'\,\bigl\|\,|f|*|v|\,\bigr\|_{L^2}^2\,|\Lambda|
    \,\leq\, C_1'\,\norm{f}_{L^1}^2\norm{v}_{L^2}^2\,|\Lambda|\,,
\end{align*}
and inserting this estimate together with \req{I2bound} and \req{R0bound}
in \req{Rv}, we finally
deduce that
\[
   \frac{{\cal R}_v}{|\Lambda|}
   \,\leq\, \bigl(2\norm{\rho_2}_{L^\infty} \,+\, 4C_0'\,\norm{f}_{L^1}
                  \,+\, C_1'\,\norm{f}_{L^1}^2\bigr)\,\norm{v}_{L^2}^2\,,
\]
which was to be shown.
\end{proofof}

Finally, we turn to the covariance of $V_\Lambda$ and $N_\Lambda$.

\begin{proofof}{Lemma~\ref{Lem:VtildeN}}
From \req{N} and \req{2ndorderintegral2}, \req{MGNZ} and \req{GNZ},
we conclude that
\begin{align*}
    \E\bigl[V_\Lambda'N_\Lambda\bigr]
    &\,=\,\E\left[
            \sum_{x_1\in\gamma}\chiL(x_1)\,
            z \int_\Lambda v(x_1-x_2) e^{-\beta W(x_2 \mid \gamma)} \dx_2 
            \sum_{x_3\in\gamma} \chiL(x_3)
            \right]     \\[1ex]
    &\,=\,\E\Biggl[
            \sum_{x_1\neq x_3\in\gamma}\chiL(x_1)\chiL(x_3)\,
            z \int_\Lambda v(x_1-x_2) e^{-\beta W(x_2 \mid \gamma)} \dx_2 
            \Biggr]  \\[1ex]
    &\phantom{\,=}
            \
          \,+\,\E\Biggl[
                 \sum_{x_1\in\gamma}\chiL(x_1)\,
                 z \int_\Lambda v(x_1-x_2) e^{-\beta W(x_2\mid\gamma)} \dx_2 
               \Biggr] \\[2ex]
    &\,=\,\E\left[
            z^3 \int_{\Lambda^3} 
                   v(x_1-x_2) e^{-\beta W(x_2\mid\gamma\cup\{x_1,x_3\})}\
                   e^{-\beta H(x_1,x_3) - \beta W(x_1,x_3\mid\gamma)}\dxx_3 
            \right] \\[1ex]
    &\phantom{\,=}
            \ +\,\E\left[
                   z^2 \int_{\Lambda^2} 
                          v(x_1-x_2) e^{-\beta W(x_2\mid\gamma\cup\{x_1\})}
                          e^{-\beta W(x_1\mid\gamma)}\dxx_2 
               \right] \,.
\end{align*}
It therefore follows from \req{rho-n} that
\be{V1N}
    \E\bigl[V_\Lambda'N_\Lambda\bigr]
    \,=\,\int_{\Lambda^3} v(x_1-x_2) \rho^{(3)}(\xx_3)\dxx_3
         \,+\, \int_{\Lambda^2} v(x_1-x_2)\rho_2(x_1-x_2)\dxx_2\,.
\ee

On the other hand
\begin{align*}
    \E\bigl[V_\Lambda'' N_\Lambda \bigr]
    &\,=\,\E\left[
            \sum_{x_3\in\gamma} \chiL(x_3)
            \int_{\Lambda^2} 
               v(x_1-x_2)\,z^2 
               e^{-\beta H(\xx_2)-\beta W(\xx_2\mid\gamma)}\dxx_2
           \right] \\[1ex]
    &\,=\,\E\left[
            z^3 \int_{\Lambda^3} 
                   v(x_1-x_2)\,
                   e^{-\beta H(\xx_2)-\beta W(\xx_2\mid\gamma\cup\{x_3\})}
                   e^{-\beta W(x_3\mid\gamma)}\dxx_3
            \right]
\end{align*}
by virtue of \req{GNZ}, and hence, \req{rho-n} yields
\be{V2N}
    \E\bigl[V_\Lambda'' N_\Lambda \bigr]
    \,=\, \int_{\Lambda^3}v(x_1-x_2)\rho^{(3)}(\xx_3)\dxx_3\,.
\ee
Accordingly, it follows from \req{Veps}, \req{V1N}, \req{V2N}, and \req{VN} 
that
\[
   \E\bigl[V_\Lambda^\eps N_\Lambda\bigr]
   \,=\, \int_{\Lambda^3} v(x_1-x_2)\rho^{(3)}(\xx_3)\dxx_3
         \,+\, 2 \int_{\Lambda^2} v(x_1-x_2)\rho_2(x_1-x_2)\dxx_2
   \,=\, \E\bigl[V_\Lambda N_\Lambda \bigr]\,.
\]
This establishes our claim \req{CovVepsN}, because $E\bigl[V_\Lambda\bigr]=E\bigl[V_\Lambda^\eps\bigr]$ by the construction
of $V_\Lambda^\eps$.
\end{proofof}


\newpage

\end{document}